\documentclass{amsart}
\usepackage{amsfonts}

\usepackage{amscd}
\usepackage{thmdefs}

\input{tcilatex}
\theoremstyle{definition}
\theoremstyle{remark}
\numberwithin{equation}{section}

\input tcilatex

\begin{document}
\title[$L^{p}$ growth and Hardy-Sobolev-Morrey inequalities]{$L^{p}$ measure of growth and higher order Hardy-Sobolev-Morrey inequalities
on $\Bbb{R}^{N}$}
\author{Patrick J. Rabier}
\address{Department of Mathematics, University of Pittsburgh, Pittsburgh, PA 15260,
USA}
\email{rabier@imap.pitt.edu}
\maketitle

\begin{abstract}
When the growth at infinity of a function $u$ on $\Bbb{R}^{N}$ is compared
with the growth of $|x|^{s}$ for some $s\in \Bbb{R},$ this comparison is
invariably made pointwise. This paper argues that the comparison can also be
made in a suitably defined $L^{p}$ sense for every $1\leq p<\infty $ and
that, in this perspective, inequalities of Hardy, Sobolev or Morrey type
account for the fact that sub $|x|^{-N/p}$ growth of $\nabla u$ in the $%
L^{p} $ sense implies sub $|x|^{1-N/p}$ growth of $u$ in the $L^{q}$ sense
for well chosen values of $q.$

By investigating how sub $|x|^{s}$ growth of $\nabla ^{k}u$ in the $L^{p}$
sense implies sub $|x|^{s+j}$ growth of $\nabla ^{k-j}u$ in the $L^{q}$
sense for (almost) arbitrary $s\in \Bbb{R}$ and for $q$ in a $p$-dependent
range of values, a family of higher order Hardy/Sobolev/Morrey type
inequalities is obtained, under optimal integrability assumptions.

These optimal inequalities take the form of estimates for $\nabla
^{k-j}(u-\pi _{u}),1\leq j\leq k,$ where $\pi _{u}$ is a suitable polynomial
of degree at most $k-1,$ which is unique if and only if $s<-k.$ More
generally, it can be chosen independent of $(s,p)$ when $s$ remains in the
same connected component of $\Bbb{R}\backslash \{-k,...,-1\}.$
\end{abstract}

\section{Introduction\label{intro}}

Unless specified otherwise, $\Bbb{R}^{N}$ is the domain of all function
spaces. If $s>-1,u\in \mathcal{D}^{\prime }$ (distributions) and $\nabla
u\in (L_{loc}^{\infty })^{N}$ grows slower than $|x|^{s}$ at infinity for
some $s>-1,$ then $u$ grows slower than $|x|^{s+1}$ at infinity. In this
statement, growth is understood pointwise, outside a set of Lebesgue measure 
$0$ and the precise result is that if $(1+|x|)^{-s}\nabla u\in (L^{\infty
})^{N},$ then $(1+|x|)^{-s-1}u\in L^{\infty }.$ This property breaks down if 
$s\leq -1$ but, if $s<-1$ and $N>1,$ it is still true that $%
(1+|x|)^{-s-1}(u-c_{u})\in L^{\infty }$ for a unique constant $c_{u}.$

The pointwise criterion is only one of the ways to compare the growth of a
function against the growth of the powers of $|x|,$ but it is not
necessarily the most useful one. For instance, it is notorious that
pointwise growth has little relevance for functions of $L_{loc}^{p}$ with $%
1\leq p<\infty ,$ and, from the context, it is intuitively clear that an $
L^{p}$ evaluation of growth could only be more adequate.

Such an $L^{p}$ measure of growth can be captured by various closely related
but non-equivalent definitions. The option chosen in this paper is to say
that $u\in L_{loc}^{p}$ grows slower than $|x|^{s}$ in the $L^{p}$ sense if $%
(1+|x|)^{-s-N/p}u\in L^{p}.$ This is justified by the remarks that the
function $u(x):=(1+|x|)^{t}$ satisfies this condition if and only if $t<s$
and that the pointwise concept is recovered when $p=\infty $ although, in
this case, $(1+|x|)^{-s}u\in L^{\infty }$ still holds if $u$ grows as fast
as $|x|^{s}$ at infinity. Strictly slower growth requires the stronger $%
\lim_{R\rightarrow \infty }\limfunc{ess}\sup_{|x|>R}|x|^{-s}|u|=0$ or simply 
$\lim_{|x|\rightarrow \infty }|x|^{-s}u(x)=0$ if $u$ is continuous. In
particular, $u\in L^{p}$ with $p<\infty $ ($p=\infty $) if and only if $u$
grows slower than $|x|^{-N/p}$ (no faster than $|x|^{0}=1$) in the $L^{p}$
sense. Of course, the choice of a scale based on $1+|x|$ rather than $|x|$
is meant to avoid integrability issues near the origin, which have nothing
to do with behavior at infinity.

It is a natural question whether the feature of the $p=\infty $ case
highlighted in the first paragraph is preserved when $p<\infty :$ If $s\neq
-1$ and $\nabla u$ grows slower than $|x|^{s}$ in the $L^{p}$ sense, is
there a constant $c_{u}$ such that $u-c_{u}$ grows slower than $|x|^{s+1}$
in the $L^{p}$ sense, or in the $L^{q}$ sense for some $q\neq p?$ If $s>-1,$
is it possible to choose $c_{u}=0?$ Although some widely explored issues,
such as inequalities of Hardy, Sobolev or Morrey type, turn out to be
intimately related to these questions, they have apparently not been tackled
up front and the connection between familiar inequalities and growth
transfer from gradient to function, while intuitively obvious, has
nonetheless remained rather vague.

This paper investigates the more general growth transfer property in the $
L^{p}-L^{q}$ sense, when $\nabla u$ is replaced with $\nabla ^{k}u$ for some 
$k\in \Bbb{N}$ and $s\in \Bbb{R}\backslash \{-k,...,-1\}.$ To deal with the
excluded values, the discussion should incorporate a logarithmic scale and
is omitted. Also, it will be necessary to assume $N>1$ when $s<-1,$ although
this restriction can be lifted when $\Bbb{R}$ is replaced with $\Bbb{R}_{\pm
}.$

The space 
\begin{equation}
L_{s}^{q}:=\{u\in L_{loc}^{q}:(1+|x|)^{-s-N/q}u\in L^{q}\},1\leq q\leq
\infty ,  \label{1}
\end{equation}
that embodies sub $|x|^{s}$ growth in the $L^{q}$ sense if $q<\infty $ (and
up to $|x|^{s}$ growth if $q=\infty $) is equipped with the Banach space
norm 
\begin{equation}
||u||_{L_{s}^{q}}:=||(1+|x|)^{-s-N/q}u||_{q},  \label{2}
\end{equation}
where $||\cdot ||_{q}$ is the $L^{q}$ norm. If $q<\infty ,$ then $%
L_{s}^{q}=L^{q}(\Bbb{R}^{N};(1+|x|)^{-sq-N}dx),$ with identical norms.

A little more notation must be introduced to give a concise summary of the
results. The number 
\begin{equation}
\nu (k,N):=\binom{N+k-1}{k},  \label{3}
\end{equation}
is the dimension of the space of real symmetric tensors of order $k\in \Bbb{N%
}$ and, for $d\in \Bbb{Z},\mathcal{P}_{d}$ denotes the space of polynomials
of degree at most $d,$ with the usual agreement that $\mathcal{P}_{d}=\{0\}$
if $d<0.$ Lastly, if $j\in \Bbb{N}$ and $1\leq p\leq \infty ,$ we set $%
p^{*j}:=Np/(N-jp)$ if $p<N/j$ and $p^{*j}:=\infty $ otherwise and 
\begin{equation}
I_{j,p}=\left\{ 
\begin{array}{l}
\lbrack p,p^{*j}]\text{ if }p\neq N/j\text{ or if }p=N=j=1,\quad \\ 
\lbrack p,\infty )\text{ if }p=N/j\text{ with }N>1.\text{ }
\end{array}
\right.  \label{4}
\end{equation}
In particular, $I_{j,p}=[p,\infty ]$ irrespective of $j$ if $N=1.$

The main result (Theorem \ref{th9}) states that if $k\in \Bbb{N},1\leq
p<\infty $ and either $s>-1$ or $N>1$ and $s\notin \{-k,...,-1\}$ and if $%
\nabla ^{k}u\in (L_{s}^{p})^{\nu (k,N)},$ there is a polynomial $\pi _{u}\in 
\mathcal{P}_{k-1}$ such that $\nabla ^{k-j}(u-\pi _{u})\in
(L_{s+j}^{q})^{\nu (k-j,N)}$ for every $1\leq j\leq k$ and every $q\in
I_{j,p}$ and there is a constant $C>0$ independent of $u$ such that 
\begin{equation}
||\,|\nabla ^{k-j}(u-\pi _{u})|\,||_{L_{s+j}^{q}}\leq C||\,|\nabla
^{k}u|\,||_{L_{s}^{p}}.  \label{5}
\end{equation}
In particular, $\nabla ^{k-j}(u-\pi _{u})$ grows slower than $|x|^{s+j}$ in
the $L^{q}$ sense for every finite $q\in I_{j,p}$ and no faster than $
|x|^{s+j}$ in the $L^{\infty }$ sense when $p>N/j$ (so that $\infty \in
I_{j,p}$). We also show that, in the latter case, $\nabla ^{k-j}(u-\pi _{u})$
still grows slower than $|x|^{s+j},$ that is, $\lim_{|x|\rightarrow \infty
}|x|^{-(s+j)}(\nabla ^{k-j}u(x)-\nabla ^{k-j}\pi _{u}(x))=0.$ When $j=k$ and 
$s=-N/p$ (i.e., $\nabla ^{k}u\in (L^{p})^{\nu (k,N)}$), this pointwise
property was proved by Mizuta \cite{Mi86}, by a different method.

If $s>-1,$ then $\mathcal{P}_{j-1}\subset L_{s+j}^{q}$ irrespective of $q$
and (\ref{5}) implies $\nabla ^{k-j}u\in L_{s+j}^{q}$ for every $q\in
I_{j,p}.$ Thus, $\pi _{u}$ is irrelevant as regards the property $\nabla
^{k-j}(u-\pi _{u})\in (L_{s+k}^{q})^{\nu (k-j,N)},$ but it remains of course
essential for the validity of (\ref{5}).

The polynomial $\pi _{u}$ may be chosen independent of $(s,p)$ when $s$
remains in any connected component of $\Bbb{R}\backslash \{-k,...,-1\}$ and
its nature is different depending upon $k$ and $s$ and, to some extent, even 
$p.$ When $s>-1,$ there are many different ways to define a (generally
different) polynomial $\pi _{u},$ each one being more or less reminiscent of
a Taylor polynomial of $u$ of order $k-1.$ That $\pi _{u}$ may be chosen as
a genuine Taylor polynomial is only true when $p>N.$ Without this
restriction (and, still, $s>-1$), the coefficients of $\pi _{u}$ can be
obtained by averaging the partial derivatives of $u$ of order up to $k-1$ on
arbitrarily chosen balls independent of $u.$ For details, see Theorem \ref
{th3} when $k=1$ and the comments following Theorem \ref{th9} in general.

In contrast, $\pi _{u}$ is unique when $s<-k$ and its coefficients depend
only upon the behavior at infinity of the partial derivatives of $u$ of
order up to $k-1.$ If $s\in (-k,-1),$ the part of $\pi _{u}$ of higher
degree is unique and depends upon the behavior of the higher order partial
derivatives of $u$ at infinity and its part of lower degree can be chosen as
a Taylor polynomial of sorts, much like in the case when $s>-1.$ Naturally,
the meaning of higher and lower degree will be clarified.

All the spaces $L_{s}^{q}$ are dilation-invariant, which allows for scaling
arguments. In many cases, scaling produces inequalities (\ref{5}) in which
the weight $1+|x|$ may be replaced with $|x|.$ When $k=1$ and $s=-N/p,$
different choices of $q$ produce the following sample of at least partially
known inequalities:

(i) $|||x|^{-1}(u-u(0))\,||_{p}\leq C||\,|\nabla u|\,||_{p}$ if $p>N$ (with $
q=p$ and $\pi _{u}=u(0)$ in (\ref{5}), plus scaling). This is Hardy's
inequality.

(ii) $||u-c_{u}||_{p^{*}}\leq C||\,|\nabla u|\,||_{p}$ for a unique constant 
$c_{u}$ if $p<N$ (with $q=p^{*}$ and $\pi _{u}=c_{u}$ in (\ref{5})), a known
generalization of Sobolev's inequality (\cite[Section 6.7.5]{Ma11}).

(iii) $\sup_{x\in \Bbb{R}^{N}}|x|^{-1+N/p}|u(x)-u(0)|\leq C||\,|\nabla
u|\,||_{p}$ if $p>N$ (with $q=\infty $ and $\pi _{u}=u(0)$ in (\ref{5}),
plus scaling). This is Morrey's inequality.

(iv) $||(1+|x|)^{-1}(u-c_{u})\,||_{p}\leq C||\,|\nabla u|\,||_{p}$ for a
unique constant $c_{u}$ if $p<N$ (with $q=p$ and $\pi _{u}=c_{u}$ in (\ref{5}%
)), a variant (with $|x|$ replaced with $1+|x|$) and generalization of the
Hardy-Leray inequality\footnote{%
That is, Hardy's inequality when $p<N.$} when $u\in \mathcal{C}_{0}^{\infty
} $ (so that $c_{u}=0$) (\cite{Le34}, \cite[Section 2.8.1]{Ma11}). By
scaling, the Hardy-Leray inequality $|||x|^{-1}(u-c_{u})\,||_{p}\leq
C||\,|\nabla u|\,||_{p}$ follows under the more general assumption $c_{u}=0.$

Other values of $q,s$ or $k$ produce inequalities of the same type. We shall
refer to the texts by Maz'ya \cite{Ma11} and Opic and Kufner \cite{OpKu90}
for various related inequalities on $\Bbb{R}^{N}$ when $N>1.$ The papers by
Caffarelli, Kohn and Nirenberg \cite{CaKoNi84}, Catrina and Costa \cite
{CaCo09}, Gatto, Guti\'{e}rrez and Wheeden \cite{GaGuWh85}, Lin \cite{Li86}
and the author \cite{Ra12}, are in a similar spirit, but specifically
devoted to inequalities involving pure power weights $|x|^{s}.$ We do not
mention work limited to Muckenhoupt weights since $(1+|x|)^{s}$ need not
belong to this class.

Aside from technical differences due to the choice of weights, the
inequalities (\ref{5}) depart from those in the above and other works in
more basic aspects. In the literature, the focus has overwhelmingly been on
inequalities of the form (\ref{5}) when $\pi _{u}=0.$ Since this is not
typical, such inequalities can only be true under restrictive assumptions.
In fact, while (\ref{5}) holds under the optimal integrability condition $%
\nabla ^{k}u\in (L_{s}^{p})^{\nu (k,N)},$ the others assume, at the very
least, that $u$ belongs to some weighted Sobolev space (\cite{OpKu90}, \cite
{Ra12}) and, much more often, $u\in \mathcal{C}_{0}^{\infty }$ or $u\in 
\mathcal{C}_{0}^{\infty }(\Bbb{R}^{N}\backslash \{0\})$ (\cite{CaKoNi84}, 
\cite{CaCo09}, \cite{GaGuWh85}, \cite{Li86}, \cite{Ma11}, \cite{OpKu90}),
especially when $k>1$ (\cite{Li86}, \cite{Ma11}). In that regard, it is
instructive to observe that if $u\in \mathcal{C}_{0}^{\infty },$ then $\pi
_{u}=0$ when $\pi _{u}$ is determined by behavior at infinity (i.e., $s<-k$)
and also $\pi _{u}=0$ if $u\in \mathcal{C}_{0}^{\infty }(\Bbb{R}
^{N}\backslash \{0\})$ and $\pi _{u}$ may be chosen as a Taylor polynomial
at $0$ (i.e., $s>-1$).

The proof of Theorem \ref{th9} is by induction on $k.$ The case when $k=1$
(Theorem \ref{th8}) is more demanding and the proof has three steps. The
first two consist in proving the theorem when $q=p$ or when $u$ is radially
symmetric and $q\in I_{1,p},$ and either $s>-1$ (Theorem \ref{th3}) or $s<-1$
(Theorem \ref{th5}). The main ingredients include a property of
approximation by mollification in weighted spaces $L^{p}(\Bbb{R}^{N};wdx)$
when $1\leq p<\infty $ and $\log w$ is uniformly continuous (Lemma \ref{lm1}%
), two special cases of well-known one-dimensional Hardy-type inequalities
(Lemmas \ref{lm2} and \ref{lm4}) and the Poincar\'{e}-Wirtinger inequality
on bounded open subsets of $\Bbb{R}^{N}$ and on the sphere $\Bbb{S}^{N-1}.$

To prove Theorem \ref{th8} \thinspace when $q>p,$ we take advantage of the
fact that the radially symmetric case has already been settled to reduce the
problem when $u$ has a vanishing radial symmetrization. Under this
additional assumption, an elaboration on an argument first used by
Caffarelli, Kohn and Nirenberg \cite{CaKoNi84} (Lemma \ref{lm7}) completes
the proof.

The existence of $\pi _{u}$ depends only upon $\nabla ^{k}u$ being in $%
(L_{s}^{p})^{\nu (k,N)},$ but further assumptions about $u$ may have an
impact on $\pi _{u}.$ In Section \ref{application}, we use this remark to
sharpen and generalize known embedding theorems of weighted Sobolev spaces.
The transfer of sub-exponential growth is briefly discussed in Section \ref
{exponential}.

Throughout the paper, $C>0$ denotes a constant whose value may change from
place to place. The notation $B_{R}$ refers to the open ball with center $0$
and radius $R>0$ in $\Bbb{R}^{N}$ and $\widetilde{B}_{R}:=\Bbb{R}
^{N}\backslash \overline{B}_{R}.$ If $1\leq p\leq \infty ,$ the H\"{o}lder
conjugate of $p$ is denoted by $p^{\prime }.$ We shall also make use of the
norms $||\cdot ||_{p,\Omega },||\cdot ||_{p,\Bbb{S}^{N-1}}$ and $||\cdot
||_{1,p,\Omega }$ of $L^{p}(\Omega ),L^{p}(\Bbb{S}^{N-1})$ and (the
classical Sobolev space) $W^{1,p}(\Omega ),$ respectively.

\section{Preliminary first order inequalities when $s>-1$ \label{case1}}

We need a property of approximation by mollification in weighted Lebesgue
spaces $L^{p}(\Bbb{R}^{N};wdx)$ when $\log w$ is uniformly continuous. Just
to put things in perspective, recall that $\log w\in BMO$ if $w$ is a
Muckenhoupt weight (\cite{Mu72}).

\begin{lemma}
\label{lm1}Let $w>0$ be a function such that $\log w$ is uniformly
continuous on $\Bbb{R}^{N}.$ The following properties hold:\newline
(i) For every $\varepsilon >0,$ there is $\delta >0$ such that $w(x)\leq %
(1+\varepsilon )w(y)$ whenever $|x-y|<\delta .$\newline
(ii) For every $\varepsilon >0,$ there is $\delta >0$ such that $|w(x)-w(y)|%
\leq \varepsilon w(y)$ whenever $|x-y|<\delta .$\newline
(iii) If $\theta _{n}\in \mathcal{C}_{0}^{\infty }$ is a sequence of
mollifiers and if $u\in L^{p}(\Bbb{R}^{N};wdx)$ for some $1\leq p<\infty ,$
then $\theta _{n}*u\in L^{p}(\Bbb{R}^{N};wdx)$ for $n$ large enough and $%
\theta _{n}*u\rightarrow u$ in $L^{p}(\Bbb{R}^{N};wdx).$
\end{lemma}

\begin{proof}
(i) Choose $\delta >0$ such that $|x-y|<\delta \Rightarrow |\log w(x)-\log
w(y)|\leq \log (1+\varepsilon ).$

(ii) If $|x-y|<\delta $ with $\delta >0$ from (i), then $w(x)-w(y)\leq
\varepsilon w(y)$ and $w(y)-w(x)\leq \varepsilon w(x)\leq \varepsilon
(1+\varepsilon )w(y).$ Thus, $|w(x)-w(y)|\leq \varepsilon (1+\varepsilon
)w(y)$ and it suffices to replace $\varepsilon (1+\varepsilon )$ with $%
\varepsilon .$

(iii) With $\varepsilon $ and $\delta $ from (i), let $n$ be large enough
that $\limfunc{Supp}\theta _{n}\subset B_{\delta /2}.$ For simplicity of
notation, set $w_{p}:=w^{1/p},$ so that $uw_{p}\in L^{p}.$ Then, 
\begin{multline*}
|((\theta _{n}*u)w_{p})(x)|=\int_{B(x,\delta /2)}\theta
_{n}(x-y)u(y)w_{p}(x)dy \\
\leq (1+\varepsilon )^{1/p}\int_{B(x,\delta /2)}\theta
_{n}(x-y)|u(y)|w_{p}(y)dy=(1+\varepsilon )^{1/p}(\theta _{n}*(|u|w_{p}))(x).
\end{multline*}
This shows that $\theta _{n}*u\in L^{p}(\Bbb{R}^{N};wdx).$ To prove that $%
\theta _{n}*u\rightarrow u$ in $L^{p}(\Bbb{R}^{N};wdx),$ i.e., that $(\theta
_{n}*u)w_{p}\rightarrow uw_{p}$ in $L^{p},$ write $(\theta
_{n}*u)w_{p}-uw_{p}=[(\theta _{n}*u)w_{p}-\theta _{n}*(uw_{p})]+[\theta
_{n}*(uw_{p})-uw_{p}].$ The latter bracket tends to $0$ in $L^{p}$ and it
suffices to prove that the same thing is true for the former.

Since $\log w_{p}=(1/p)\log w$ is uniformly continuous on $\Bbb{R}^{N},$
part (ii) is applicable to $w_{p}.$ Thus, given $\varepsilon >0,$ if $\delta
>0$ is small enough and if $n$ is large enough that $\limfunc{Supp}\theta
_{n}\subset B_{\delta /2},$ 
\begin{multline*}
|((\theta _{n}*u)w_{p})(x)-(\theta _{n}*(uw_{p}))(x)|\leq \int_{B(x,\delta
/2)}\theta _{n}(x-y)|u(y)||w_{p}(x)-w_{p}(y)|dy \\
\leq \varepsilon \int_{B(x,\delta /2)}\theta
_{n}(x-y)|u(y)|w_{p}(y)dy=\varepsilon (\theta _{n}*(|u|w_{p}))(x).
\end{multline*}
As a result, $||(\theta _{n}*u)w_{p}-\theta _{n}*(uw_{p})||_{p}\leq
\varepsilon ||\theta _{n}*(|u|w_{p})||_{p}.$ Since the right-hand side tends
to $\varepsilon ||uw_{p}||_{p}$ and $\varepsilon >0$ is arbitrary, $\lim
||(\theta _{n}*u)w_{p}-\theta _{n}*(uw_{p})||_{p}=0$ and the proof is
complete.
\end{proof}

\begin{remark}
\label{rm1}Obviously, Lemma \ref{lm1} is valid when $w(x)=(1+|x|)^{a}$ or $%
w(x)=e^{a|x|}$ and $a\in \Bbb{R}.$
\end{remark}

We shall also need a special case of a known one-dimensional weighted Hardy
inequality. Lemma \ref{lm2} below follows from Bradley \cite[Theorem 1]{Br78}
or Maz'ya \cite[p. 40 ff]{Ma11}. Since the weights $r^{t}$ and $(1+r)^{t}$
are equivalent on $[\rho ,\infty )$ with $\rho >0,$ it also follows directly
from Opic and Kufner \cite[Example 6.9, p. 70]{OpKu90} when $q<\infty .$

\begin{lemma}
\label{lm2} Suppose that $s>-1$ and that $1\leq p<\infty $ and let $\rho >0$
be given. \newline
(i) If $p\leq q<\infty ,$ there is a constant $C>0$ such that 
\begin{multline}
\left( \int_{\rho }^{\infty }(1+r)^{-(s+1)q-N}r^{N-1}|f(r)-f(\rho
)|^{q}dr\right) ^{1/q}  \label{6} \\
\leq C\left( \int_{\rho }^{\infty }(1+r)^{-sp-N}r^{N-1}|f^{\prime
}(r)|^{p}dr\right) ^{1/p},
\end{multline}
for every locally absolutely continuous function $f$ on $[\rho ,\infty ).$ 
\newline
(ii) There is a constant $C>0$ such that 
\begin{equation}
\sup_{r\geq R}(1+r)^{-(s+1)}|f(r)-f(R)|\leq C\left( \int_{R}^{\infty %
}(1+r)^{-sp-N}r^{N-1}|f^{\prime }(r)|^{p}dr\right) ^{1/p},  \label{7}
\end{equation}
for every $R\geq \rho $ and every locally absolutely continuous function $f$
on $[\rho ,\infty )$.
\end{lemma}

In (\ref{7}), the result when $R>\rho $ follows from the case when $R=\rho $
with $f$ replaced with $f_{R}=f(R)$ on $[\rho ,R)$ and $f_{R}=f$ on $%
[R,\infty ),$ so that $f_{R}(\rho )=f(R),f_{R}^{\prime }=0$ on $[\rho ,R)$
and $f_{R}^{\prime }=f^{\prime }$ on $(R,\infty ).$ This does not affect $C.$

\begin{theorem}
\label{th3}Suppose that $s>-1$ and that $1\leq p<\infty $. If $u\in \mathcal{
D}^{\prime }$ and $\nabla u\in (L_{s}^{p})^{N},$ set 
\begin{equation*}
c_{u}:=|B_{\rho }|^{-1}\int_{B_{\rho }}u,
\end{equation*}
where $\rho >0$ is chosen once and for all and independent of $u.$ Then:%
\newline
(i) $u\in L_{s+1}^{p}$ and there is a constant $C=C(s,p)>0$ independent of $u
$ such that 
\begin{equation}
||u-c_{u}||_{L_{s+1}^{p}}\leq C||\,|\nabla u|\,||_{L_{s}^{p}}.\newline
\label{8}
\end{equation}
(ii) If $N=1$ or if $u$ is radially symmetric, $u\in L_{s+1}^{q}$ for every $%
q\in I_{1,p}$ (see (\ref{4})) and there is a constant $C=C(s,p,q)>0$
independent of $u$ such that 
\begin{equation}
||u-c_{u}||_{L_{s+1}^{q}}\leq C||\,|\nabla u|\,||_{L_{s}^{p}}.  \label{9}
\end{equation}
Furthermore, if $p>N$ or $p=N=1$ (so that $\infty \in I_{1,p}$), $%
\lim_{|x|\rightarrow \infty }|x|^{-(s+1)}u(x)=0.$
\end{theorem}

\begin{proof}
Suppose first that $u\in \mathcal{C}^{\infty }$ and let $u=u(r,\sigma )$
with $r\geq 0$ and $\sigma \in \Bbb{S}^{N-1}.$ If $p\leq q<\infty ,$ it
follows from (\ref{6}) that 
\begin{multline*}
\int_{\rho }^{\infty }(1+r)^{-(s+1)q-N}r^{N-1}|u(r,\sigma )-u(\rho ,\sigma
)|^{q}dr \\
\leq C\left( \int_{\rho }^{\infty }(1+r)^{-sp-N}r^{N-1}|\partial
_{r}u(r,\sigma )|^{p}dr\right) ^{q/p},
\end{multline*}
for every $\sigma \in \Bbb{S}^{N-1},$ where $\partial _{r}u$ is the radial
derivative of $u.$ Since $|u(r,\sigma )|^{q}\leq 2^{q-1}[|u(r,\sigma
)-u(\rho ,\sigma )|^{q}+|u(\rho ,\sigma )|^{q}]$ and since $\int_{\rho
}^{\infty }(1+r)^{-(s+1)q-N}r^{N-1}dr<\infty $ (recall $s>-1$), we infer
that 
\begin{multline}
\int_{\rho }^{\infty }(1+r)^{-(s+1)q-N}r^{N-1}|u(r,\sigma )|^{q}dr
\label{10} \\
\leq C\left[ |u(\rho ,\sigma )|^{q}+\left( \int_{\rho }^{\infty
}(1+r)^{-sp-N}r^{N-1}|\partial _{r}u(r,\sigma )|^{p}dr\right) ^{q/p}\right]
\\
\leq C\left[ |u(\rho ,\sigma )|^{q}+\left( \int_{\rho }^{\infty
}(1+r)^{-sp-N}r^{N-1}|\nabla u(r,\sigma )|^{p}dr\right) ^{q/p}\right] .
\end{multline}
(i) If $q=p$ above, integration on $\Bbb{S}^{N-1}$ yields 
\begin{equation}
||u||_{L_{s+1}^{p}(\widetilde{B}_{\rho })}\leq C(||u(\rho ,\cdot )||_{p,\Bbb{
S}^{N-1}}+||\,|\nabla u|\,||_{L_{s}^{p}(\widetilde{B}_{\rho })}),  \label{11}
\end{equation}
when $u\in \mathcal{C}^{\infty }.$ Suppose now that $u\in \mathcal{D}
^{\prime }$ and that $\nabla u\in (L_{s}^{p})^{N}.$ Let $\theta _{n}$ denote
a mollifying sequence and set $u_{n}:=\theta _{n}*u.$ By Remark \ref{rm1}
and Lemma \ref{lm1} (iii), $\nabla u_{n}=\theta _{n}*\nabla u\rightarrow
\nabla u$ in $(L_{s}^{p})^{N}.$ In particular, $\nabla u_{n}\rightarrow
\nabla u$ in $(L_{s}^{p}(\widetilde{B}_{\rho }))^{N}.$ On the other hand,
since $\nabla u\in (L_{s}^{p})^{N}\subset (L_{loc}^{p})^{N},$ then $u\in
W_{loc}^{1,p}$ (\cite[p. 21]{Ma11}). Thus, $u_{n}\rightarrow u$ in $%
W_{loc}^{1,p}$ which, by the continuity of the trace (even when $p=1;$ see 
\cite[p. 164]{AdFo03}) implies $u_{n}(\rho ,\cdot )\rightarrow u(\rho ,\cdot
)$ in $L^{p}(\Bbb{S}^{N-1}).$ As a result, 
\begin{equation*}
\lim_{n,m\rightarrow \infty }||u_{n}(\rho ,\cdot )-u_{m}(\rho ,\cdot )||_{p,%
\Bbb{S}^{N-1}}+||\,|\nabla (u_{n}-u_{m})|\,||_{L_{s}^{p}(\widetilde{B}_{\rho
})}=0
\end{equation*}
and so, by (\ref{11}), $u_{n}$ is a Cauchy sequence in $L_{s+1}^{p}(%
\widetilde{B}_{\rho }).$ Call $v$ its limit, so that $u_{n}\rightarrow v$ in 
$L_{loc}^{1}(\widetilde{B}_{\rho }).$ Since also $u_{n}\rightarrow u$ in $%
W_{loc}^{1,p}\hookrightarrow L_{loc}^{1},$ it follows that $u=v\in
L_{s+1}^{p}(\widetilde{B}_{\rho })$ and (\ref{11}) holds. The remark that $%
L_{s+1}^{p}(B_{\rho })=L^{p}(B_{\rho })$ now yields $u\in L_{s+1}^{p}.$

It remains to prove (\ref{8}). Upon replacing $u$ with $u-c_{u},$ it is not
restrictive to assume $\int_{B_{\rho }}u=0,$ so that $c_{u}=0.$ By the
continuity of the trace, $||u(\rho ,\cdot )||_{p,\Bbb{S}^{N-1}}\leq
C||u||_{1,p,B_{\rho }}$ and, by the Poincar\'{e}-Wirtinger inequality, the
seminorm $||\,|\nabla u|\,||_{p,B_{\rho }}$ is equivalent to the norm $%
||u||_{1,p,B_{\rho }}$ on the subspace of functions of $W^{1,p}(B_{\rho })$
with zero mean. Thus, $||u(\rho ,\cdot )||_{p,\Bbb{S}^{N-1}}\leq
C||\,|\nabla u|\,||_{p,B_{\rho }}$ and $||u||_{p,B_{\rho }}\leq C||\,|\nabla
u|\,||_{p,B_{\rho }}.$ Since $(1+|x|)^{t}$ is bounded above and below on $%
B_{\rho }$ for every $t,$ it follows that 
\begin{equation*}
||u||_{L_{s+1}^{p}(B_{\rho })}\leq C||\,|\nabla u|\,||_{L_{s}^{p}}\text{ and 
}||u||_{L_{s+1}^{p}(\widetilde{B}_{\rho })}\leq C||\,|\nabla
u|\,||_{L_{s}^{p}},
\end{equation*}
where (\ref{11}) was used to obtain the second inequality. This proves (\ref
{8}) when $\int_{B_{\rho }}u=0$ and, hence, in general.

(ii) Suppose $q\in I_{1,p}$ and either $N=1$ or $u$ is radially symmetric.
We only discuss the latter case since it will be clear that the former can
be handled similarly. Since $\nabla u\in (L_{s}^{p})^{N}$ implies that $u\in
W_{loc}^{1,p}$ is a function, there is no need to introduce a distribution
definition of radial symmetry.

If $q<\infty ,$ the proof of (\ref{9}) proceeds as in (i), with only minor
modifications. When $u\in \mathcal{C}^{\infty },$ (\ref{10}) is still valid
but, since now $u$ is radially symmetric, both $u$ and $\partial _{r}u$
depend only upon $r$ and the inequality becomes 
\begin{multline*}
\int_{\rho }^{\infty }(1+r)^{-(s+1)q-N}r^{N-1}|u(r)|^{q}dr \\
\leq C\left[ |u(\rho )|^{q}+\left( \int_{\rho }^{\infty
}(1+r)^{-sp-N}r^{N-1}|\partial _{r}u(r)|^{p}dr\right) ^{q/p}\right] .
\end{multline*}
Up to a factor independent of $u,$ the integrals $\int_{\rho }^{\infty
}(1+r)^{-(s+1)q-N}r^{N-1}|u(r)|^{q}dr$ and $\int_{\rho }^{\infty
}(1+r)^{-sp-N}r^{N-1}|\partial _{r}u(r)|^{p}dr$ are $||u||_{L_{s+1}^{q}(%
\widetilde{B}_{\rho })}^{q}$ and $||\partial _{r}u||_{L_{s}^{p}(\widetilde{B}
_{\rho })}^{p},$ respectively. Hence, 
\begin{equation*}
||u||_{L_{s+1}^{q}(\widetilde{B}_{\rho })}\leq C(|u(\rho )|+||\,|\nabla
u|\,||_{L_{s}^{p}(\widetilde{B}_{\rho })}).
\end{equation*}
Up to another factor independent of $u,$ the number $|u(\rho )|$ is the $%
L^{p}(\Bbb{S}^{N-1})$ norm of the constant function $u(\rho ).$ Therefore,
since $W^{1,p}(B_{\rho })\hookrightarrow L^{q}(B_{\rho })$ for $q\in I_{1,p}$
and since the radial symmetry is preserved in approximations $u_{n}=\theta
_{n}*u$ by simply choosing radially symmetric mollifiers, the proof can be
completed exactly as before.

If $p>N$ or $p=N=1,$ then $\infty \in I_{1,p}$ and the proof when $q=\infty $
is similar: Just use (\ref{7}) with $R=\rho $ instead of (\ref{6}) to get $%
||u||_{L_{s+1}^{\infty }(\widetilde{B}_{\rho })}\leq C(|u(\rho
)|+||\,|\nabla u|\,||_{L_{s}^{p}(\widetilde{B}_{\rho })}).$ (The
approximation by mollification is only used in $L_{s}^{p}$ with $p<\infty .$)

To see that, in addition, $\lim_{|x|\rightarrow \infty }|x|^{-(s+1)}u(x)=0,$
observe first that, by radial symmetry, $u(x)=f_{u}(|x|)$ with $f_{u}$
locally absolutely continuous on $(0,\infty )$ and $f_{u}^{\prime
}(|x|)=\partial _{r}u(x)$ (for more details, see the proof of Lemma \ref{lm6}
later), so that, by (\ref{7}), 
\begin{multline*}
\sup_{r\geq R}(1+r)^{-(s+1)}|f_{u}(r)-f_{u}(R)| \\
\leq C\left( \int_{R}^{\infty }(1+r)^{-sp-N}r^{N-1}|f_{u}^{\prime
}(r)|^{p}dr\right) ^{1/p}\leq C||\,|\nabla u|\,||_{L_{s}^{p}(\widetilde{B}%
_{R})},
\end{multline*}
where $C>0$ is also independent of $R.$ Thus, if $r\geq R,$%
\begin{equation*}
(1+r)^{-(s+1)}|f_{u}(r)|\leq (1+r)^{-(s+1)}|f_{u}(R)|+C||\,|\nabla
u|\,||_{L_{s}^{p}(\widetilde{B}_{R})}.
\end{equation*}
Given $\varepsilon >0,$ choose $R$ large enough that $C||\,|\nabla
u|\,||_{L_{s}^{p}(\widetilde{B}_{R})}\leq \varepsilon .$ Since $s>-1,$ it
follows that $\lim \sup_{r\rightarrow \infty }(1+r)^{-(s+1)}|f_{u}(r)|\leq
\varepsilon $ and so $\lim_{r\rightarrow \infty }(1+r)^{-(s+1)}f_{u}(r)=0.$
Equivalently, $\lim_{r\rightarrow \infty }r^{-(s+1)}f_{u}(r)=0,$ whence $%
\lim_{|x|\rightarrow \infty }|x|^{-(s+1)}u(x)=0.$
\end{proof}

The choice $c_{u}=$ $|B_{\rho }|^{-1}\int_{B\rho }u$ in (\ref{8})\ and (\ref
{9}) is not the only possible one. By translation, we may choose $%
c_{u}=|B_{\rho }|^{-1}\int_{B(x_{0},\rho )}u$ where $x_{0}\in \Bbb{R}^{N}$
is arbitrary but independent of $u.$ The constant $C$ depends upon $\rho $
and $x_{0},$ but it does not necessarily blow up as $\rho \rightarrow 0:$

\begin{remark}
\label{rm2}If $p>N$ or if $N=1$ (and in no other case), the obvious variants
of the inequalities (\ref{6}) and (\ref{7}) continue to hold on $[0,\infty )$
(\cite[Theorem 5.9, p. 63]{OpKu90}) and minor modifications of the proof of
Theorem \ref{th3} yield $||u-u(0)||_{L_{s+1}^{p}}\leq C||\,|\nabla
u|\,||_{L_{s}^{p}}$ for every $u\in \mathcal{D}^{\prime }$ such that $\nabla
u\in (L_{s}^{p})^{N}$ as well as $||u-u(0)||_{L_{s+1}^{q}}\leq C||\,|\nabla
u|\,||_{L_{s}^{p}}$ if $N=1$ or if $u$ is radially symmetric and $q\in
I_{1,p}=[p,\infty ].$ Naturally, $u(0)$ may also be replaced with $u(x_{0})$
where $x_{0}$ is independent of $u.$
\end{remark}

\section{Preliminary first order inequalities when $s<-1$\label{case2}}

We begin with a different version of Lemma \ref{lm2}, a special case of 
\cite[Theorem 6.2, p. 65]{OpKu90} that can also be found in \cite[Theorem 2]
{Br78} or \cite[p. 40 ff]{Ma11}. However, the proof is sketched to show how
the restriction $q\in I_{1,p}$ (not needed in Lemma \ref{lm2}) arises.

\begin{lemma}
\label{lm4}Suppose that $s<-1$ and that $1\leq p<\infty .$ \newline
(i) For every finite $q\in I_{1,p}$ (see (\ref{4})), there is a constant $C>0
$ such that 
\begin{multline}
\left( \int_{0}^{\infty }(1+r)^{-(s+1)q-N}r^{N-1}|f(r)|^{q}dr\right) ^{1/q}
\label{12} \\
\leq C\left( \int_{0}^{\infty }(1+r)^{-sp-N}r^{N-1}|f^{\prime
}(r)|^{p}dr\right) ^{1/p},
\end{multline}
for every locally absolutely continuous function $f$ on $(0,\infty )$ such
that $\lim_{r\rightarrow \infty }f(r)=0.$\newline
(ii) If $p>N$ (so that $\infty \in I_{1,p}$), there is a constant $C>0$ such
that 
\begin{equation}
\sup_{r\geq R}(1+r)^{-(s+1)}|f(r)|\leq C\left( \int_{R}^{\infty %
}(1+r)^{-sp-N}r^{N-1}|f^{\prime }(r)|^{p}dr\right) ^{1/p},  \label{13}
\end{equation}
for every $R\geq 0$ and every locally absolutely continuous function $f$ on $%
(0,\infty )$ such that $\lim_{r\rightarrow \infty }f(r)=0.$
\end{lemma}

\begin{proof}
From \cite[Theorem 6.2, p. 65]{OpKu90}, (\ref{12}) and (\ref{13}) with $R=0$
hold when $q\geq p$ if and only if $\sup_{\xi >0}A(\xi )B(\xi )<\infty ,$
where $A(\xi ):=||(1+r)^{-(s+1)-N/q}r^{(N-1)/q}||_{q,(0,\xi )}$ and $B(\xi
):=||(1+r)^{s+N/p}r^{(1-N)/p}||_{p^{\prime },(\xi ,\infty )}.$ Thus,
everything boils down to showing that $A(\xi )B(\xi )$ is bounded when $\xi
\rightarrow 0$ and when $\xi \rightarrow \infty .$ Note that $B(\xi )<\infty 
$ since $s<-1.$

If $\xi >0$ is small, a routine verification shows that $A(\xi )=O(\xi
^{N/q})$ and that $B(\xi )=O(1)$ if $p>N,B(\xi )=O(|\log \xi |^{(N-1)/N})$
if $p=N$ and $B(\xi )=O(\xi ^{1-N/p})$ if $p<N.$ Thus, $A(\xi )B(\xi )$ is
bounded near $0$ if $p>N$ or $p=N=1,$ or if $p=N$ and $q<\infty ,$ or if $%
p<N $ and $N/q+1-N/p\geq 0,$ i.e., $q\leq p^{*}.$ In other words, $A(\xi
)B(\xi ) $ is bounded near the origin if and only if $q\in I_{1,p}.$ For
large $\xi ,$ $A(\xi )=O(\xi ^{-(s+1)})$ and $B(\xi )=O(\xi ^{s+1}),$ so
that $A(\xi )B(\xi )$ is always bounded.

In (\ref{13}), the result when $R>0$ follows from the case when $R=0$ with $%
f $ replaced with $f_{R}=f(R)$ on $[0,R)$ and $f_{R}=f$ on $[R,\infty ),$ so
that $f_{R}^{\prime }=0$ on $[0,R)$ and $f_{R}^{\prime }=f^{\prime }$ on $%
(R,\infty ).$
\end{proof}

\begin{theorem}
\label{th5}Suppose $N>1,$ $s<-1$ and $1\leq p<\infty .$ If $u\in \mathcal{D}%
^{\prime }$ and $\nabla u\in (L_{s}^{p})^{N},$ then:\newline
(i) There is a unique constant $c_{u}\in \Bbb{R}$ such that $u-c_{u}\in
L_{s+1}^{p}$ and there is a constant $C=C(s,p)>0$ independent of $u$ such
that 
\begin{equation}
||u-c_{u}||_{L_{s+1}^{p}}\leq C||\,|\nabla u|\,||_{L_{s}^{p}}.  \label{14}
\end{equation}
\newline
(ii) If also $u$ is radially symmetric, then for every $q\in I_{1,p},$ $c_{u}
$ in (i) is the unique constant such that $u-c_{u}\in L_{s+1}^{q}$ and there
is a constant $C=C(s,p,q)>0$ independent of $u$ such that 
\begin{equation}
||u-c_{u}||_{L_{s+1}^{q}}\leq C||\,|\nabla u|\,||_{L_{s}^{p}}.  \label{15}
\end{equation}
Furthermore, if $p>N$ (so that $\infty \in I_{1,p}$), $\lim_{|x|\rightarrow 
\infty }|x|^{-(s+1)}(u(x)-c_{u})=0.$
\end{theorem}

\begin{proof}
The uniqueness of $c_{u}$ is obvious since $L_{s+1}^{q}$ contains no nonzero
constant when $s<-1$ and $1\leq q\leq \infty .$ We focus on the existence
part. Some preliminary properties must be established to prove parts (i) and
(ii) of the theorem.

It is well-known that if $u\in W_{loc}^{1,1},$ then $u$ is locally
absolutely continuous on almost every line parallel to the coordinate axes $%
x_{i}$ and that, on such lines, the classical and weak derivatives $\partial
_{i}u$ coincide. Together with the local equivalence of the measures $dr$
and $r^{N-1}dr$ away from the origin, this implies that, when passing to
spherical coordinates, $u(\cdot ,\sigma )$ is locally absolutely continuous
on $(0,\infty )$ (but not necessarily on $[0,\infty )$) for a.e. $\sigma \in 
\Bbb{S}^{N-1},$ with classical radial derivative $\partial _{r}u(r,\sigma
)=\nabla u(r,\sigma )\cdot \sigma .$ In particular, this holds if $\nabla
u\in (L_{s}^{p})^{N}.$

From now on, we assume $s<-1$ and $\nabla u\in (L_{s}^{p})^{N},$ so that $%
(1+|x|)^{-s-N/p}\partial _{r}u\in L^{p}.$ By Fubini's theorem in spherical
coordinates, $(1+r)^{-s-N/p}r^{(N-1)/p}\partial _{r}u(\cdot ,\sigma )\in
L^{p}(0,\infty )$ for a.e. $\sigma \in \Bbb{S}^{N-1}.$ Since $%
(1+r)^{s+N/p}r^{(1-N)/p}\in L^{p^{\prime }}(\varepsilon ,\infty )$ for every 
$\varepsilon >0$ when $s<-1,$ it follows that $\partial _{r}u(\cdot ,\sigma
)\in L^{1}(\varepsilon ,\infty ).$ Consequently, 
\begin{equation*}
v(r,\sigma ):=\int_{\infty }^{r}\partial _{r}u(t,\sigma )dt,
\end{equation*}
is a.e. defined and measurable on $(0,\infty )\times \Bbb{S}^{N-1}.$ For
a.e. $\sigma \in \Bbb{S}^{N-1},$ the function $v(\cdot ,\sigma )$ is locally
absolutely continuous and a.e. differentiable on $(0,\infty )$ with $%
\partial _{r}v(\cdot ,\sigma )=\partial _{r}u(\cdot ,\sigma )$ and $%
\lim_{r\rightarrow \infty }v(r,\sigma )=0.$ In particular, 
\begin{equation}
c_{u}(\sigma ):=u(\cdot ,\sigma )-v(\cdot ,\sigma ),  \label{16}
\end{equation}
is a function independent of $r>0$ (difference of two locally absolutely
continuous functions with the same a.e. derivative).

Next, $v(r,\cdot )\in L^{p}(\Bbb{S}^{N-1})$ for every $r>0$ and $%
\lim_{r\rightarrow \infty }||v(r,\cdot )||_{p,\Bbb{S}^{N-1}}=0.$ To see
this, use the estimate 
\begin{equation*}
|v(r,\sigma )|\leq \lambda (r)\left( \int_{r}^{\infty
}(1+t)^{-sp-N}t^{N-1}|\partial _{r}u(t,\sigma )|^{p}\right) ^{1/p},
\end{equation*}
where $\lambda (r):=||(1+t)^{s+N/p}t^{(1-N)/p}||_{p^{\prime },(r,\infty
)}\rightarrow 0$ when $r\rightarrow \infty .$ By taking $p^{th}$ powers and
integrating on $\Bbb{S}^{N-1},$ we get $||v(r,\cdot )||_{p,\Bbb{S}%
^{N-1}}\leq \lambda (r)^{1/p^{\prime }}||\,|\nabla
u|\,||_{L_{s}^{p}}\rightarrow 0$ when $r\rightarrow \infty ,$ as claimed.
Thus, by (\ref{16}), 
\begin{equation}
\lim_{r\rightarrow \infty }||u(r,\cdot )-c_{u}||_{p,\Bbb{S}^{N-1}}=0.
\label{17}
\end{equation}

The next step is to show that $c_{u}$ is actually constant. (When $1<p<N$
and $\nabla u\in (L^{p})^{N},$ i.e., $s=-N/p,$ this goes back to Uspenskii 
\cite{Us61}; see also Fefferman \cite{Fe74}.) We shall use the
Poincar\'{e}-Wirtinger inequality on the sphere $\Bbb{S}^{N-1}:$ If $N>1,$ 
\begin{equation}
||w-\overline{w}||_{p,\Bbb{S}^{N-1}}\leq C||\nabla _{\Bbb{S}^{N-1}}w||_{p,%
\Bbb{S}^{N-1}},  \label{18}
\end{equation}
for every $w\in W^{1,p}(\Bbb{S}^{N-1}),$ where $\nabla _{\Bbb{S}^{N-1}}$ is
the gradient of $w$ for the natural Riemannian structure of the unit sphere, 
$C>0$ is a constant independent of $w$ and $\overline{w}$ is the average of $%
w$ on $\Bbb{S}^{N-1}.$ In the literature, the Poincar\'{e}-Wirtinger
inequality on compact manifolds is mostly quoted when $p=2$ (Osserman \cite
{Os78}), but an elementary proof for arbitrary $p$ follows, by
contradiction, from the connectedness of $\Bbb{S}^{N-1}$ and the compactness
of the embedding $W^{1,p}(\Bbb{S}^{N-1})\hookrightarrow L^{p}(\Bbb{S}
^{N-1}). $

Assume $u\in \mathcal{C}^{\infty }$ (in addition to $\nabla u\in
(L_{s}^{p})^{N}$). When $r>0$ is fixed, $\nabla _{\Bbb{S}^{N-1}}u(r,\sigma )$
is the orthogonal projection of $\nabla u(r,\sigma )$ on the tangent space $%
\{\sigma \}^{\bot }$ of $\Bbb{S}^{N-1}$ at $\sigma ,$ whence $|\nabla _{\Bbb{%
S}^{N-1}}u(r,\sigma )|\leq |\nabla u(r,\sigma )|.$ Thus, by (\ref{18}), $%
||u(r,\cdot )-\overline{u}(r)||_{p,\Bbb{S}^{N-1}}\leq C||\,|\nabla u(r,\cdot
)|\,||_{p,\Bbb{S}^{N-1}}$ where $\overline{u}(r)$ is the average of $%
u(r,\cdot )$ on $\Bbb{S}^{N-1}$ and so 
\begin{multline*}
\int_{0}^{\infty }(1+r)^{-sp-N}r^{N-1}||u(r,\cdot )-\overline{u}(r)||_{p,%
\Bbb{S}^{N-1}}^{p}dr \\
\leq C\int_{0}^{\infty }(1+r)^{-sp-N}r^{N-1}||\,|\nabla u(r,\cdot )|\,||_{p,%
\Bbb{S}^{N-1}}^{p}dr=C||\,|\nabla u|\,||_{L_{s}^{p}}^{p}.
\end{multline*}
Since the left-hand side is finite, there is a sequence $r_{n}\rightarrow
\infty $ such that $\lim (1+r_{n})^{-sp-N}r_{n}^{N-1}||u(r_{n},\cdot )-%
\overline{u}(r_{n})||_{p,\Bbb{S}^{N-1}}=0,$ which in turn implies $\lim
||u(r_{n},\cdot )-\overline{u}(r_{n})||_{p,\Bbb{S}^{N-1}}=0$ because $\lim
(1+r_{n})^{-sp-N}r_{n}^{N-1}=\infty $ when $s<-1.$ Together with $\lim
||u(r_{n},\cdot )-c_{u}||_{p,\Bbb{S}^{N-1}}=0$ from (\ref{17}), this yields $
\lim ||\overline{u}(r_{n})-c_{u}||_{p,\Bbb{S}^{N-1}}=0$ and, since $%
\overline{u}(r_{n})$ is independent of $\sigma ,$ it follows that $c_{u}$ is
constant (under the additional assumption $u\in \mathcal{C}^{\infty }$ at
this point).

We are now in a position to prove (i) and (ii) of the theorem.

(i) Recall that $\lim_{r\rightarrow \infty }v(r,\sigma )=0$ for a.e. $\sigma
\in \Bbb{S}^{N-1}.$ Since $v(r,\sigma )=u(r,\sigma )-c_{u}(\sigma ),$ the
choice $f(r)=u(r,\sigma )-c_{u}(\sigma )$ and $q=p$ in (\ref{12}) yields 
\begin{multline*}
\int_{0}^{\infty }(1+r)^{-(s+1)p-N}r^{N-1}|u(r,\sigma )-c_{u}(\sigma )|^{p}dr
\\
\leq C\int_{0}^{\infty }(1+r)^{-sp-N}r^{N-1}|\partial _{r}u(r,\sigma
)|^{p}dr,
\end{multline*}
whence, by integration on $\Bbb{S}^{N-1},$%
\begin{equation}
||u-c_{u}||_{L_{s+1}^{p}}\leq C||\,|\nabla u|\,||_{L_{s}^{p}}.  \label{19}
\end{equation}
Set $u_{n}:=\theta _{n}*u$ where $\theta _{n}$ is a mollifying sequence. By
Lemma \ref{lm1} for $\nabla u$ and Remark \ref{rm1}, it follows from (\ref
{19}) that $u_{n}-c_{u_{n}}$ is a Cauchy sequence in $L_{s+1}^{p}.$ Call $%
\widetilde{u}$ its limit. Then, $u_{n}-c_{u_{n}}\rightarrow \widetilde{u}$
in $L_{loc}^{1}$ and, since $u_{n}\rightarrow u$ in $L_{loc}^{1},$ we infer
that $c_{u_{n}}\rightarrow u-\widetilde{u}$ in $L_{loc}^{1}.$ From the
above, $c_{u_{n}}$ is constant because $u_{n}\in \mathcal{C}^{\infty }.$
Therefore, $u-\widetilde{u}$ is a constant $\widetilde{c}$ and $u-\widetilde{%
c}=\widetilde{u}\in L_{s+1}^{p}.$ Thus, by (\ref{19}), $c_{u}-\widetilde{c}%
\in L_{s+1}^{p}$ and, since neither $c_{u}$ nor $\widetilde{c}$ depends upon 
$r$ and $\int_{0}^{\infty }(1+r)^{-(s+1)p-N}r^{N-1}dr=\infty $ when $s<-1,$
this can only happen if $c_{u}=\widetilde{c}$ a.e. on $\Bbb{S}^{N-1}.$ This
shows that $c_{u}$ is constant and so (\ref{19}) is the inequality (\ref{14}
).

(ii) If $u$ is radially symmetric and $q\in I_{1,p}$ is finite, it follows
from (\ref{12}) with $f(r)=u(r)-c_{u}$ that 
\begin{multline*}
\left( \int_{0}^{\infty }(1+r)^{-(s+1)q-N}r^{N-1}|u(r)-c_{u}|^{q}dr\right)
^{1/q} \\
\leq C\left( \int_{0}^{\infty }(1+r)^{-sp-N}r^{N-1}|\partial
_{r}u(r)|^{p}dr\right) ^{1/p},
\end{multline*}
which, up to a constant factor independent of $u,$ is just the inequality (%
\ref{15}).

If $p>N,$ the same inequality when $q=\infty $ follows from (\ref{13}) with $
R=0$ instead of (\ref{12}). In addition, by (\ref{13}) with $R>0,$ we also
get $||u-c_{u}||_{L_{s+1}^{\infty }(\widetilde{B}_{R})}\leq C||\,|\nabla
u|\,||_{L_{s}^{p}(\widetilde{B}_{R})}$ with $C>0$ independent of $R$ and so $%
\lim_{R\rightarrow \infty }||u-c_{u}||_{L_{s+1}^{\infty }(\widetilde{B}%
_{R})}=\lim_{R\rightarrow \infty }||\,|\nabla u|\,||_{L_{s}^{p}(\widetilde{B}%
_{R})}=0.$ By the continuity of $u$ (recall $p>N$), this amounts to $%
\lim_{|x|\rightarrow \infty }|x|^{-(s+1)}(u(x)-c_{u})=0$.
\end{proof}

By (\ref{17}) and since $L^{p}(\Bbb{S}^{N-1})$ $\hookrightarrow L^{1}(\Bbb{S}%
^{N-1}),$ it follows that

\begin{equation}
c_{u}=\lim_{r\rightarrow \infty }(N\omega _{N})^{-1}\int_{\Bbb{S}%
^{N-1}}u(r,\sigma )d\sigma ,  \label{20}
\end{equation}
where $u(r,\cdot )$ is the trace of $u$ on $\partial B_{r}$ and $\omega _{N}$
is the measure of the unit ball of $\Bbb{R}^{N}.$ In particular, $c_{u}$ is
independent of $s<-1$ and $1\leq p<\infty $ such that $\nabla u\in
(L_{s}^{p})^{N}.$

\begin{remark}
\label{rm3}Although Theorem \ref{th5} is false when $N=1$ (and indeed (\ref
{18}) breaks down), it is readily checked that it remains true on $\Bbb{R}_{%
\pm }.$ Its failure on $\Bbb{R}$ is only due to the fact that the
restrictions of $u$ to $\Bbb{R}_{-}$ and $\Bbb{R}_{+}$ need not involve the
same constant $c_{u}.$
\end{remark}

\section{The general inequalities\label{general}}

In this section, Theorems \ref{th3} and \ref{th5} are complemented and
subsumed in a single statement (Theorem \ref{th8}). Next, the result is
generalized when $\nabla ^{k}u\in (L_{s}^{p})^{\nu (N,k)}$ for some integer $%
k\in \Bbb{N}$ (Theorem \ref{th9}).

Recall that $\omega _{N}$ is the measure of the unit ball of $\Bbb{R}^{N}$
and suppose $u\in L_{loc}^{p}$ with $1\leq p<\infty .$ By Fubini's theorem
in spherical coordinates, 
\begin{equation*}
f_{u}(t):=(N\omega _{N})^{-1}\int_{\Bbb{S}^{N-1}}u(t\sigma )d\sigma ,
\end{equation*}
is defined for a.e. $t>0$ and $f_{u}\in L_{loc}^{p}([0,\infty
),t^{N-1}dt)\subset L_{loc}^{p}(0,\infty ).$ The radial symmetrization $%
u_{S} $ of $u$ is the radially symmetric function 
\begin{equation*}
u_{S}(x):=f_{u}(|x|)=(N\omega _{N})^{-1}\int_{\Bbb{S}^{N-1}}u(|x|\sigma
)d\sigma .
\end{equation*}

\begin{lemma}
\label{lm6}If $u\in \mathcal{D}^{\prime }$ and $\nabla u\in (L_{s}^{p})^{N}$
with $s\in \Bbb{R}$ and $1\leq p<\infty ,$ then $\nabla u_{S}\in
(L_{s}^{p})^{N}$ and $||\,|\nabla u_{S}|\,||_{L_{s}^{p}}\leq ||\,|\nabla
u|\,||_{L_{s}^{p}}.$
\end{lemma}

\begin{proof}
First, $u\in W_{loc}^{1,p}$ since $\nabla u\in (L_{loc}^{p})^{N}.$ We claim
that $u_{S}\in W_{loc}^{1,p},$ which is obvious if $N=1.$ If $N>1,$ then $%
W^{1,p}(B_{R})=W^{1,p}(B_{R}\backslash \{0\})$ (see for instance\textit{\ }%
\cite[p. 52]{HeKiMa93}) and it suffices to show that $u_{S}\in
W^{1,p}(B_{R}\backslash \{0\})$ for every $R>0.$ That $u_{S}\in
L^{p}(B_{R}\backslash \{0\})=L^{p}(B_{R})$ is clear from $u\in
W_{loc}^{1,p}. $ Since $\partial _{r}u=\nabla u\cdot |x|^{-1}x\in
L_{loc}^{p},$ the formal calculation 
\begin{equation}
\nabla u_{S}(x)=(N\omega _{N})^{-1}\left( \int_{\Bbb{S}^{N-1}}\partial
_{r}u(|x|\sigma )d\sigma \right) |x|^{-1}x,  \label{21}
\end{equation}
yields $\nabla u_{S}\in (L^{p}(B_{R})\backslash \{0\})^{N}(L^{p}(B_{R}))^{N}$
and so, as claimed, $u_{S}\in W^{1,p}(B_{R}\backslash \{0\}).$ This formula
is justified below when $\nabla u_{S}$ is understood as a distribution on $%
\Bbb{R}^{N}\backslash \{0\},$ but since $W^{1,p}(B_{R}\backslash
\{0\})=W^{1,p}(B_{R}),$ it also gives $\nabla u_{S}$ as a distribution on $%
\Bbb{R}^{N}.$

If $\varphi \in \mathcal{C}_{0}^{\infty }(0,\infty ),$ set $\psi
(x):=\varphi (|x|),$ so that $\psi \in \mathcal{C}_{0}^{\infty }(\Bbb{R}%
^{N}\backslash \{0\})$ and that $\partial _{r}\psi (x)=\varphi ^{\prime
}(|x|).$ Then, $\langle f_{u}^{\prime },\varphi \rangle =-(N\omega
_{N})^{-1}\left\langle u,|x|^{1-N}\partial _{r}\psi \right\rangle =(N\omega
_{N})^{-1}\left\langle |x|^{1-N}\partial _{r}u,\psi \right\rangle $ (use $%
\partial _{r}=|x|^{-1}x\cdot \nabla $ and $\nabla \cdot \left(
|x|^{-N}x\right) =0$). Since $\partial _{r}u\in L_{loc}^{p},$ this shows
that $\langle f_{u}^{\prime },\varphi \rangle =\langle f_{\partial
_{r}u},\varphi \rangle ,$ that is, $f_{u}^{\prime }=f_{\partial _{r}u}\in
L_{loc}^{p}(0,\infty )$ and so $f_{u}\in W_{loc}^{1,p}(0,\infty ).$ In
particular, $f_{u}$ is locally absolutely continuous on $(0,\infty ).$ As a
result, by Marcus and Mizel \cite[Theorem 4.3]{MaMi72}, $\nabla
u_{S}(x)=|x|^{-1}f_{u}^{\prime }(|x|)x=|x|^{-1}f_{\partial _{r}u}(|x|)x$ as
a distribution on $\Bbb{R}^{N}\backslash \{0\}$ and (\ref{21}) is proved.

To see that $\nabla u_{S}\in (L_{s}^{p})^{N},$ use (\ref{21}) and
H\"{o}lder's inequality to get $|\nabla u_{S}(x)|\leq (N\omega
_{N})^{-1}\int_{\Bbb{S}^{N-1}}|\nabla u(|x|\sigma )|d\sigma \leq (N\omega
_{N})^{-1/p}\left( \int_{\Bbb{S}^{N-1}}|\nabla u(|x|\sigma )|^{p}d\sigma
\right) ^{1/p}.$ Hence,\linebreak $(1+|x|)^{-sp-N}|\nabla u_{S}(x)|^{p}\leq
(N\omega _{N})^{-1}\int_{\Bbb{S}^{N-1}}(1+|x|)^{-sp-N}|\nabla u(|x|\sigma
)|^{p}d\sigma $ and so, by integration in spherical coordinates, $%
||\,|\nabla u_{S}|\,||_{L_{s}^{p}}\leq ||\,|\nabla u|\,||_{L_{s}^{p}}.$
\end{proof}

If $u\in \mathcal{D}^{\prime }$ and $\nabla u\in (L_{s}^{p})^{N},$ Lemma \ref
{lm6} yields $\nabla u_{S}\in (L_{s}^{p})^{N}$ and it then follows from
Theorem \ref{th3} (Theorem \ref{th5}) that $u_{S}\in L_{s+1}^{q}$ for every $%
q\in I_{1,p}$ if $s>-1$ ($u_{S}-c_{u_{S}}\in L_{s+1}^{q}$ for every $q\in
I_{1,p}$ and a unique constant $c_{u_{S}}$ if $s<-1$). Thus, to show that $%
u\in L_{s+1}^{q}$ or that $u-c_{u}\in L_{s+1}^{q}$ when $q>p,$ it suffices
to prove the same result for $u-u_{S}.$ The difference between $u$ and $%
u-u_{S}$ is that $(u-u_{S})_{S}=0$ and that, for functions with vanishing
radial symmetrization, a result originating in the work of Caffarelli, Kohn
and Nirenberg \cite{CaKoNi84} and generalized in \cite{Ra12} is applicable.
We only spell out the special case relevant to the issue of interest here
and give a proof of it when $q=\infty ,$ not considered elsewhere.

\begin{lemma}
\label{lm7}Suppose that $N>1.$ If $u\in L_{loc}^{1}$ and $u_{S}=0$ and if $%
|x|^{a/p}u\in L^{p}$ and $|x|^{1+a/p}\nabla u\in (L^{p})^{N}$ for some $a\in 
\Bbb{R}$ and $1\leq p<\infty ,$ there is a constant $C=C(a,p,q)>0$
independent of $u$ such that 
\begin{equation}
||\,|x|^{(a+N)/p-N/q}u||_{q}\leq C||\,|x|^{1+a/p}|\nabla u|\,||_{p},
\label{22}
\end{equation}
for every $q\in I_{1,p}$ (see (\ref{4})). Furthermore, if $p>N,$ then $%
\lim_{|x|\rightarrow \infty }|x|^{(a+N)/p}u(x)=0.$
\end{lemma}

\begin{proof}
If $q<\infty ,$ the result follows by letting $b=a+p$ and by substituting $%
q=p,r=q$ in part (ii) of \cite[Corollary 6.1]{Ra12}. The proof when $%
q=\infty $ (hence $p>N$) is given below.

For $\tau >0,$ let $\Omega _{\tau }:=\{x\in \Bbb{R}^{N}:\tau <|x|<2\tau \}.$
Since power weights are bounded above and below on $\Omega _{1},$ it follows
that $u\in W^{1,p}(\Omega _{1})\hookrightarrow L^{\infty }(\Omega _{1}).$
Thus, $||u||_{\infty ,\Omega _{1}}\leq C||u||_{1,p,\Omega _{1}}$ with a
constant $C>0$ independent of $u.$

The assumption $u_{S}=0$ entails $\int_{\Omega _{\tau }}u=0$ for every $\tau
>0$ and so, since $\Omega _{1}$ is connected when $N>1,$ $||u||_{\infty
,\Omega _{1}}\leq C||\,|\nabla u|\,||_{p,\Omega _{1}}$ by the
Poincar\'{e}-Wirtinger inequality. Once again by the boundedness of power
weights on $\Omega _{1},$ this yields $||\,|x|^{(a+N)/p}u||_{\infty ,\Omega
_{1}}\leq C||\,|x|^{1+a/p}|\nabla u|\,||_{p,\Omega _{1}}.$ More generally,
by scaling, 
\begin{equation}
||\,|x|^{(a+N)/p}u||_{\infty ,\Omega _{\tau }}\leq C||\,|x|^{1+a/p}|\nabla
u|\,||_{p,\Omega _{\tau }},  \label{23}
\end{equation}
with the same $C>0$ independent of $\tau .$ The right-hand side is majorized
by $C||\,|x|^{1+a/p}|\nabla u|\,||_{p}$ and then (\ref{22}) when $q=\infty $
follows from $||\,|x|^{(a+N)/p}u||_{\infty }=\sup_{\tau
>0}||\,|x|^{(a+N)/p}u||_{\infty ,\Omega _{\tau }}.$ The proof of this
equality is a simple exercise.

More generally, $||v||_{\infty ,\widetilde{B}_{R}}=\sup_{\tau \geq
R}||\,v||_{\infty ,\Omega _{\tau }}.$ Thus, when $v=\,|x|^{(a+N)/p}u$ with $
u $ as above, $||\,|x|^{(a+N)/p}u||_{\infty ,\widetilde{B}_{R}}\leq
C||\,|x|^{1+a/p}|\nabla u|\,||_{p,\widetilde{B}_{R}}$ by (\ref{23}). Since $%
p<\infty ,$ $\lim_{R\rightarrow \infty }||\,|x|^{(a+N)/p}u||_{\infty ,%
\widetilde{B}_{R}}=\lim_{R\rightarrow \infty }||\,|x|^{1+a/p}|\nabla
u|\,||_{p,\widetilde{B}_{R}}=0.$ By the continuity of $u$ away from $0$
(recall $p>N$) this means $\lim_{|x|\rightarrow \infty }|x|^{(a+N)/p}u(x)=0.$
\end{proof}

\begin{theorem}
\label{th8} Suppose that $s\neq -1$ and that $1\leq p<\infty $ and let $u\in 
\mathcal{D}^{\prime }$ be such that $\nabla u\in (L_{s}^{p})^{N}.$ \newline
(i) If $s>-1,$ then $u\in L_{s+1}^{q}$ for every $q\in I_{1,p}$ (see (\ref{4}%
)) and there are a constant $c_{u}$ independent of $s$ and $p$ (the same as
in Theorem \ref{th3}) and a constant $C=C(s,p,q)>0$ independent of $u$ such
that 
\begin{equation}
||u-c_{u}||_{L_{s+1}^{q}}\leq C||\,|\nabla u|\,||_{L_{s}^{p}},  \label{24}
\end{equation}
Furthermore, if $p>N$ or $p=N=1,$ then $\lim_{|x|\rightarrow \infty %
}|x|^{-(s+1)}u(x)=0.$\newline
(ii) If $s<-1$ and $N>1,$ there is a unique constant $c_{u}\in \Bbb{R}$
independent of $s$ and $p$ (the same as in Theorem \ref{th5}) such that $%
u-c_{u}\in L_{s+1}^{q}$ for every $q\in I_{1,p}$ and there is a constant $%
C=C(s,p,q)>0$ independent of $u$ such that (\ref{24}) holds. Furthermore, if 
$p>N,$ then $\lim_{|x|\rightarrow \infty }|x|^{-(s+1)}(u(x)-c_{u})=0.$
\end{theorem}

\begin{proof}
It is obvious that the constant $c_{u}=|B_{\rho }|^{-1}\int_{B_{\rho }}u$ of
Theorem \ref{th3} is independent of $s>-1$ and $p.$ For the constant $c_{u}$
of Theorem \ref{th5}, this independence of $s<-1$ and $p$ was noticed in the
comments following that theorem. Also, in (ii), the uniqueness of $c_{u}$
follows from $L_{s+1}^{q}$ containing no nonzero constant irrespective of $q$
when $s<-1.$

If $q=p,$ or if $u$ is radially symmetric, or if $N=1$ in (i), everything
was proved in Theorems \ref{th3} and \ref{th5}. Accordingly, we henceforth
assume $N>1$ and $q\in I_{1,p}.$ The formula $c_{u}=|B_{\rho
}|^{-1}\int_{B_{\rho }}u$ if $s>-1$ (Theorem \ref{th3}) shows that $%
c_{u}=c_{u_{S}}$ and, by (\ref{20}), the same thing is true if $s<-1.$ Thus, 
$||u-c_{u}||_{L_{s+1}^{q}}\leq
||u-u_{S}||_{L_{s+1}^{q}}+||u_{S}-c_{u_{S}}||_{L_{s+1}^{q}}$ and, since the
theorem is true in the radially symmetric case, it follows from Lemma \ref
{lm6} that $||u_{S}-c_{u_{S}}||_{L_{s+1}^{q}}\leq C||\,|\nabla
u|\,||_{L_{s}^{p}}.$ Consequently, the proof of (\ref{24}) is reduced to
showing that $||u-u_{S}||_{L_{s+1}^{q}}\leq C||\,|\nabla u|\,||_{L_{s}^{p}}.$
Since $||\,|\nabla (u-u_{S})|\,||_{L_{s}^{p}}\leq 2||\,|\nabla
u|\,||_{L_{s}^{p}}$ by Lemma \ref{lm6}, it suffices to show that $%
||u-u_{S}||_{L_{s+1}^{q}}\leq C||\,|\nabla (u-u_{S})|\,||_{L_{s}^{p}}.$ From
the remark that $(u-u_{S})_{S}=0,$ this will follow from 
\begin{equation}
||u||_{L_{s+1}^{q}}\leq C||\,|\nabla u|\,||_{L_{s}^{p}},  \label{25}
\end{equation}
when $\nabla u\in (L_{s}^{p})^{N}$ and $u_{S}=0.$

Likewise, since $\lim_{|x|\rightarrow \infty
}|x|^{-(s+1)}(u_{S}(x)-c_{u_{S}})=0$ when $p>N$ is known (by radial symmetry
and Theorems \ref{th3} and \ref{th5}) and since $c_{u}=c_{u_{S}},$ the proof
that $\lim_{|x|\rightarrow \infty }|x|^{-(s+1)}u(x)=0$ or that $%
\lim_{|x|\rightarrow \infty }|x|^{-(s+1)}(u(x)-c_{u})=0$ when $p>N$ is
reduced to showing that $\lim_{|x|\rightarrow \infty
}|x|^{-(s+1)}(u(x)-u_{S}(x))=0,$ i.e. that $\lim_{|x|\rightarrow \infty
}|x|^{-(s+1)}u(x)=0$ when $\nabla u\in (L_{s}^{p})^{N}$ and $u_{S}=0.$

From now on, $\nabla u\in (L_{s}^{p})^{N}$ and $u_{S}=0.$ We shall make
repeated use, without further mention, of the elementary properties that for
every $t\in \Bbb{R},$ the weights $(1+|x|)^{t}$ are bounded above and below
on every bounded subset of $\Bbb{R}^{N}$ and that they are equivalent to $%
|x|^{t}$ when $|x|$ is bounded away from $0.$

By Theorem \ref{th3}, $u\in L_{s+1}^{p}$ if $s>-1$ and, by Theorem \ref{th5}%
, $u\in L_{s+1}^{p}$ if $s<-1$ because $c_{u}=c_{u_{S}}$ and $u_{S}=0$ show
that $c_{u}=0.$ Thus, $u\in L_{s+1}^{p}$ when $s\neq -1.$

Let $\varphi \in \mathcal{C}^{\infty }$ be radially symmetric, with $\varphi
=0$ on a neighborhood of $0$ and $\varphi =1$ outside $B_{1}.$ Since $\nabla
u\in (L_{s}^{p})^{N}$ and $u\in L_{s+1}^{p},$ it is readily checked that $%
\varphi u\in L_{s+1}^{p}$ and $\nabla (\varphi u)\in (L_{s}^{p})^{N}$ and
that $(\varphi u)_{S}=0.$ By Lemma \ref{lm7} with $a=-(s+1)p-N,$ we infer
that $|x|^{-(s+1)-N/q}\varphi u\in L^{q}$ for every $q\in I_{1,p}$ and that 
\begin{equation*}
||\,|x|^{-(s+1)-N/q}\varphi u||_{q}\leq C||\,|x|^{-s-N/p}|\nabla (\varphi
u)|\,||_{p}.
\end{equation*}
From the equivalence of weights away from $0$ and since $\varphi =1$ outside 
$B_{1},$ this implies $\lim_{|x|\rightarrow \infty }u(x)=0$ if $p>N$ (by
Lemma \ref{lm7}) and, irrespective of $p,$ 
\begin{equation}
||u||_{L_{s+1}^{q}(\widetilde{B}_{1})}\leq C||\,|\nabla (\varphi
u)|\,||_{L_{s}^{p}}\leq C||u||_{p,\Omega }+C||\,|\nabla u|\,||_{L_{s}^{p}},
\label{26}
\end{equation}
where $\Omega $ is an annulus centered at the origin (not a ball, so that $%
|x|^{-sp-N}$ is bounded above on $\Omega $) containing $\limfunc{Supp}\nabla
\varphi .$

Note that $u_{S}=0$ implies $\int_{\Omega }u=0.$ Thus, since $\Omega $ is
connected, $||u||_{p,\Omega }\leq C||\,|\nabla u|\,||_{p,\Omega }\leq
C||\,|\nabla u|\,||_{L_{s}^{p}}$ by the Poincar\'{e}-Wirtinger inequality
and so, by (\ref{26}), 
\begin{equation}
||u||_{L_{s+1}^{q}(\widetilde{B}_{1})}\leq C||\,|\nabla u|\,||_{L_{s}^{p}}.
\label{27}
\end{equation}

On the other hand, $||u||_{L_{s+1}^{q}(B_{1})}\leq C||u||_{q,B_{1}}$ and $%
||u||_{q,B_{1}}\leq C||u||_{1,p,B_{1}}$ since $q\in I_{1,p}$ and $u\in
W_{loc}^{1,p}.$ In addition, $\int_{B_{1}}u=0$ and so $%
||u||_{L_{s+1}^{q}(B_{1})}\leq C||\,|\nabla u|\,||_{p,B_{1}}\leq
C||\,|\nabla u|\,||_{L_{s}^{p}}$ by the Poincar\'{e}-Wirtinger inequality on 
$B_{1}.$ Together with (\ref{27}), this proves (\ref{25}).
\end{proof}

\begin{remark}
\label{rm4}If $s>-1$ and $p>N,$ one may also choose $c_{u}=u(0)$ in (\ref{24}
). See Remark \ref{rm2} and notice that since $u$ is continuous, $%
u_{S}(0)=u(0)$ if $u_{S}$ is extended by continuity at $0.$ Thus, the
property $c_{u}=c_{u_{S}}$ is preserved (in particular, $c_{u}=0$ if $u_{S}=0
$) and the above proof can be repeated verbatim. Once again, by translation,
one may also choose $c_{u}=u(x_{0})$ with $x_{0}\in \Bbb{R}^{N}$ independent
of $u.$
\end{remark}

When $s>-1$ and $p>N,$ (\ref{24}) with $q\in I_{1,p}=[p,\infty ]$ and $%
c_{u}=u(0)$ (Remark \ref{rm4}) reads $||(1+|x|)^{-(s+1)-N/q}(u-u(0))||_{q}%
\leq C||\,(1+|x|)^{-s-N/p}|\nabla u|\,||_{p}.$ This inequality for $%
u_{\lambda }(x)=u(\lambda x),$ $\lambda >0,$ yields $||(\lambda
+|x|)^{-(s+1)-N/q}(u-u(0))||_{q}\leq C||(\lambda +|x|)^{-s-N/p}\,|\nabla
u|\,||_{p}$ (same $C$) and so $||\,|x|^{-(s+1)-N/q}(u-u(0))||_{q}\leq
C||(|x|^{-s-N/p}\,|\nabla u|\,||_{p}$ by Fatou's lemma and monotone
convergence ($s>-N/p$) or dominated convergence ($s\leq -N/p$). If $s<-N/p,$
then $|x|^{-s-N/p}\,|\nabla u|\in L^{p}$ does not imply $\nabla u\in
(L_{s}^{p})^{N}$ unless $\nabla u\in (L_{loc}^{p})^{N},$ which must then be
assumed. Hardy's (Morrey's) inequality is recovered when $s=-N/p$ ($>-1$)
and $q=p$ ($q=\infty $).

When $s>-1$ and $p\leq N$ (hence $s>-N/p$) and if $0$ is in the Lebesgue set
of $u,$ the constant $c_{u_{\lambda }}=|B_{\rho }|^{-1}\int_{B_{\rho
}}u_{\lambda }=|B_{\lambda \rho }|^{-1}\int_{B_{\lambda \rho }}u$ tends to
some finite value $\bar{u}(0)$ independent of $\rho $ as $\lambda
\rightarrow 0.$ Then, by scaling, $||\,|x|^{-(s+1)-N/q}(u-\bar{u}(0))||_{q}%
\leq C||(|x|^{-s-N/p}\,|\nabla u|\,||_{p}$ follows as before. This extends
the previous inequality when $p>N$ and $\bar{u}(0)=u(0),$ but $s=-N/p$ (the
classical case) and $q=\infty \notin I_{1,p}$ are now ruled out.

When $s<-1$ and $N>1,$ (\ref{24}) is $||(1+|x|)^{-(s+1)-N/q}(u-c_{u})||_{q}%
\leq C||(1+|x|)^{-s-N/p}\,|\nabla u|\,||_{p}$ with $c_{u}$ now given by (\ref
{20}). If $1\leq p<N,q=p^{*}$ and $s=-N/p$ ($<-1$), this is Sobolev's
inequality $||u-c_{u}||_{p^{*}}\leq C||\,|\nabla u|\,||_{p}.$ If $c_{u}=0,$
then $c_{u_{\lambda }}=0$ by (\ref{20}) and scaling yields $%
||\,|x|^{-(s+1)-N/q}u||_{q}\leq C||\,|x|^{-s-N/p}|\nabla u|\,||_{p}$ for $%
q\in I_{1,p},$ a general Hardy-Sobolev inequality. Once again, $\nabla u\in
(L_{loc}^{p})^{N}$ must be assumed if $s<-N/p.$ When $u\in \mathcal{C}%
_{0}^{\infty }$ (hence $c_{u}=0$ and $\nabla u\in (L_{loc}^{p})^{N}$) and $%
1\leq p<N,$ another proof is given by Maz'ya (case $m=N,n=0$ in [\cite{Ma11}%
, Corollary 2, p. 139]). If $q=p<N$ and $s=-N/p$ ($<-1$), the Hardy-Leray
inequality $||\,|x|^{-1}u||_{p}\leq C||\,|\nabla u|\,||_{p}$ is recovered.
\bigskip

The next theorem generalizes Theorem \ref{th8}. Before stating it, a
cautionary remark is in order. If, in Theorem \ref{th8}, $\nabla u\in
(L_{s_{1}}^{p_{1}})^{N}\cap (L_{s_{2}}^{p_{2}})^{N}$ and $s_{1}<-1<s_{2},$
both parts (i) and (ii) of the theorem are applicable. This yields two
constants $c_{u,i}$ independent of $s_{i},i=1,2,$ with $c_{u,1}$ unique,
such that (\ref{24}) holds with $s=s_{i}$ and $q\in I_{1,p_{i}}.$ Although
there are many ways to define $c_{u,2}$ as a function of $u,$ there is no
reason why $c_{u,2}=c_{u,1}$ would be an admissible choice whenever both
constants exist. Indeed, in Theorem \ref{th8}, the constant $c_{u}$ is only
independent of $s$ in each connected component of $\Bbb{R}\backslash \{-1\}$
and its definition must be changed when $s$ crosses $-1.$

A similar issue arises if $s_{1},s_{2}>-1$ and $p_{1}\leq N<p_{2}.$ If so,
it is possible to choose $c_{u,1}=c_{u,2}=|B_{\rho }|^{-1}\int_{B_{\rho }}u$
with $\rho >0.$ However, by Remark \ref{rm2}, $c_{u,2}=u(0)$ is another
possible choice, but since $p_{1}\leq N,$ this does not mean that $%
c_{u,1}=u(0)$ is admissible.

\begin{theorem}
\label{th9}Suppose that $k\in \Bbb{N},$ that $s\notin \{-k,...,-1\}$ and
that $N>1$ if $s<-1.$ Let $u\in \mathcal{D}^{\prime }$ be such that $\nabla
^{k}u\in (L_{s}^{p})^{\nu (k,N)}$ with $1\leq p<\infty .$ Then, there is a
polynomial $\pi _{u}\in \mathcal{P}_{k-1},$ independent of $p$ and
independent of $s$ in each connected component of $\Bbb{R}\backslash
\{-k,...,-1\},$ unique if $s<-k,$ such that $\nabla ^{k-j}(u-\pi _{u})\in
(L_{s+j}^{q})^{\nu (k-j,N)}$ for every $1\leq j\leq k$ and every $q\in
I_{j,p}$ (see (\ref{4})) and there is a constant $C=C(s,j,p,q)>0$
independent of $u$ such that 
\begin{equation}
||\,|\nabla ^{k-j}(u-\pi _{u})|\,||_{L_{s+j}^{q}}\leq C||\,|\nabla
^{k}u|\,||_{L_{s}^{p}}.  \label{28}
\end{equation}
Furthermore, $\lim_{|x|\rightarrow \infty }|x|^{-(s+j)}(\nabla
^{k-j}u(x)-\nabla ^{k-j}\pi _{u}(x))=0$ if $p=N=j=1$ or if $1\leq j\leq k$
and $p>N/j.$ (In particular, $\lim_{|x|\rightarrow \infty %
}|x|^{-(s+j)}\nabla ^{k-j}u(x)=0$ if also $s>-1.$)
\end{theorem}

\begin{proof}
The uniqueness of $\pi _{u}$ when $s<-k$ follows from the remark that $%
L_{s+k}^{q}$ contains no nonzero polynomial for any $q.$

By Theorem \ref{th8}, $\pi _{u}=c_{u}$ exists when $k=1.$ Suppose $k>1$ and
that $\pi _{u}$ exists when $k$ is replaced with $k-1.$ The hypothesis $%
\nabla ^{k}u\in (L_{s}^{p})^{\nu (k,N)}$ implies $\nabla (\partial ^{\alpha
}u)\in (L_{s}^{p})^{N}$ for every multi-index $\alpha $ with $|\alpha |=k-1.$
Since $s\neq -1,$ it follows from Theorem \ref{th8} that there is a constant 
$c_{\alpha }:=c_{\partial ^{\alpha }u},$ independent of $s$ in each
connected component of $\Bbb{R}\backslash \{-1\}$ (and independent of $p,$
but we will return to this point later), such that $\partial ^{\alpha
}u-c_{\alpha }\in L_{s+1}^{p_{1}}$ for every $p_{1}\in I_{1,p}$ and there is
a constant $C_{\alpha }>0$ independent of $u$ such that $||\partial ^{\alpha
}u-c_{\alpha }||_{L_{s+1}^{p_{1}}}\leq C_{\alpha }||\,|\nabla (\partial
^{\alpha }u)|\,||_{L_{s}^{p}}.$ Upon replacing $C_{\alpha }$ with $%
\max_{\alpha }C_{\alpha },$ this yields 
\begin{equation}
||\partial ^{\alpha }u-c_{\alpha }||_{L_{s+1}^{p_{1}}}\leq C||\,|\nabla
^{k}u|\,||_{L_{s}^{p}},  \label{29}
\end{equation}
with $C$ independent of $u$ and $\alpha .$ For future use, note also that,
still by Theorem \ref{th8}, 
\begin{equation}
\lim_{|x|\rightarrow \infty }|x|^{-(s+1)}(\partial ^{\alpha }u(x)-c_{\alpha
})=0\text{ if }p>N\text{ or }p=N=1.  \label{30}
\end{equation}
Set 
\begin{equation}
\pi _{u,k-1}(x):=\sum_{|\alpha |=k-1}(\alpha !)^{-1}c_{\alpha }x^{\alpha }
\label{31}
\end{equation}
and let $v:=u-$ $\pi _{u,k-1}.$ Then, $\partial ^{\alpha }v=\partial
^{\alpha }u-c_{\alpha }$ for every $\alpha $ with $|\alpha |=k-1$ and, by (%
\ref{29}), $\nabla ^{k-1}v\in (L_{s+1}^{p_{1}})^{\nu (k-1,N)}$ for every $%
p_{1}\in I_{1,p},$ with 
\begin{equation}
||\,|\nabla ^{k-1}v|\,||_{L_{s+1}^{p_{1}}}\leq C||\,|\nabla
^{k}u|\,||_{L_{s}^{p}}.  \label{32}
\end{equation}

Since $s\notin \{-k,...,-1\}$ implies $s+1\notin \{-k+1,...,-1\},$ it
follows from the hypothesis of induction with $s$ replaced with $s+1$ that,
as long as $p_{1}$ above is finite, there is a polynomial $\pi _{v}\in 
\mathcal{P}_{k-2}$ independent of $p_{1}$ and independent of $s+1$ in each
connected component of $\Bbb{R}\backslash \{-k+1,...,-1\},$ such that $%
\nabla ^{k-1-j}(v-\pi _{v})\in (L_{s+1+j}^{q})^{\nu (k-1,N)}$ for every $%
1\leq j\leq k-2$ and every $q\in I_{j,p_{1}}$ and that, for every such $j$
and $q,$ there is a constant $C=C(j,s,p_{1},q)>0$ independent of $v$ such
that $||\,|\nabla ^{k-1-j}(v-\pi _{v})|\,||_{L_{s+1+j}^{q}}\leq C||\,|\nabla
^{k-1}v|\,||_{L_{s+1}^{p_{1}}}.$

Upon changing $j$ into $j-1,$ this may be rewritten as 
\begin{equation}
||\,|\nabla ^{k-j}(v-\pi _{v})|\,||_{L_{s+j}^{q}}\leq C||\,|\nabla
^{k-1}v|\,||_{L_{s+1}^{p_{1}}},  \label{33}
\end{equation}
for $2\leq j\leq k$ and $q\in \cup _{p_{1}\in I_{1,p},p_{1}<\infty
}I_{j-1,p_{1}}.$ A routine verification shows that $\cup _{p_{1}\in
I_{1,p},p_{1}<\infty }I_{j-1,p_{1}}=I_{j,p}.$ Thus, (\ref{33}) holds for $%
2\leq j\leq k$ and $q\in I_{j,p}.$ By (\ref{32}) and since $v-\pi _{v}=u-\pi
_{u}$ with $\pi _{u}:=\pi _{v}+\pi _{u,k-1},$ it follows that 
\begin{equation*}
||\,|\nabla ^{k-j}(u-\pi _{u})|\,||_{L_{s+j}^{q}}\leq C||\,|\nabla
^{k}u|\,||_{L_{s}^{p}},
\end{equation*}
for $2\leq j\leq k$ and $q\in I_{j,p}.$ Since (\ref{29}) is the same
inequality when $j=1$ (with $q$ called $p_{1}\in I_{1,p}$), the proof of (%
\ref{28}) is complete.

As noted, $\pi _{u,k-1}$ is independent of $s$ in each connected components
of $\Bbb{R}\backslash \{-1\}$ and $\pi _{v}$ is independent of $s+1$ in each
connected components of $\Bbb{R}\backslash \{-k+1,...,-1\},$ that is, of $s$
in each connected components of $\Bbb{R}\backslash \{-k,...,-2\}.$ Thus, $%
\pi _{u}:=\pi _{v}+\pi _{u,k-1}$ is independent of $s$ in each connected
component of $\Bbb{R}\backslash \{-k,...,-1\}.$ That $\pi _{u}$ is also
independent of $p$ will be obvious when we discuss how $\pi _{u}$ can be
calculated, after Remark \ref{rm6}.

We now prove the ``furthermore'' part. By (\ref{30}), $\lim_{|x|\rightarrow
\infty }|x|^{-(s+j)}(\nabla ^{k-j}u(x)-\nabla ^{k-j}\pi _{u}(x))=0$ if $%
p>N/j $ or $p=N=1$ holds if $j=1.$ In general, the proof goes by induction
on $j\in \{1,...,k\}.$ Suppose $j>1$ (hence $k>1$) and $p>N/j.$ Then, $%
p^{*}>N/(j-1)$ (recall $p^{*}=\infty $ if $p\geq N$ and $p^{*}=Np/(N-p)$ if $%
p<N$) and so the interval $I_{1,p}$ contains some $p_{1}\in (N/(j-1),\infty
).$ Therefore, with $v$ as above, it follows from the hypothesis of
induction with $s$ replaced with $s+1$ and $j$ replaced with $j-1$ (and
since $(s+1)+(j-1)=s+j$) that $\lim_{|x|\rightarrow \infty
}|x|^{-(s+j)}(v(x)-\pi _{v}(x))=0.$ Since $v-\pi _{v}=u-\pi _{u}$ by
definition of $\pi _{u},$ this is $\lim_{|x|\rightarrow \infty
}|x|^{-(s+j)}(\nabla ^{k-j}u(x)-\nabla ^{k-j}\pi _{u}(x))=0.$

If $s>-1,$ then $-(s+j)<1-j.$ Hence, $\lim_{|x|\rightarrow \infty
}|x|^{-(s+j)}\nabla ^{k-j}\pi _{u}(x)=0$ and so $\lim_{|x|\rightarrow \infty
}|x|^{-(s+j)}\nabla ^{k-j}u(x)=0.$
\end{proof}

\begin{remark}
\label{rm5}By Remark \ref{rm3}, Theorem \ref{th8} and Theorem \ref{th9}
remain true on $\Bbb{R}_{\pm },$ even if $s<-1.$
\end{remark}

\begin{remark}
\label{rm6}A generalization of Theorem \ref{th9}, with the same proof, is as
follows: If $\ell \in \{1,...,k\}$ and $s\notin \{-\ell ,...,-1\}$ and if $
N>1$ when $s<-1,$ a polynomial $\pi _{u}\in \mathcal{P}_{k-1}$ still exists,
which is unique modulo\footnote{%
This simply means that its part of degree at least $k-\ell $ is unique.} $%
\mathcal{P}_{k-\ell -1}$ if $s<-\ell ,$ such that $\nabla ^{k-j}(u-\pi
_{u})\in (L_{s+j}^{q})^{\nu (k-j,N)}$ and (\ref{28}) holds for every $1\leq
j\leq \ell $ and every $q\in I_{j,p}.$ Furthermore, $\lim_{|x|\rightarrow
\infty }|x|^{-(s+j)}(\nabla ^{k-j}u(x)-\nabla ^{k-j}\pi _{u}(x))=0$ if $%
p=N=j=1$ or if $1\leq j\leq \ell $ and $p>N/j.$ The part of $\pi _{u}$ of
degree at least $k-\ell $ is obtained as in the proof of Theorem \ref{th9}
and its part of degree at most $k-\ell -1$ is irrelevant since it vanishes
under the action of $\nabla ^{k-j}$ when $j\leq \ell .$ Theorem \ref{th9} is
recovered when $\ell =k.$
\end{remark}

Further comments about the polynomial $\pi _{u}$ of Theorem \ref{th9} are in
order. If $s<-k,$ then $s<-1$ and the coefficients $(\alpha !)^{-1}c_{\alpha
}$ of $\pi _{u,k-1}$ in (\ref{31}) are given by the formula (\ref{20}) with $%
u$ replaced with $(\alpha !)^{-1}\partial ^{\alpha }u$ and $|\alpha |=k-1.$
If $k>1,$ finding $\pi _{u}=\pi _{u,k-1}+\pi _{v}$ amounts to finding $\pi
_{v}$ where $v=u-\pi _{u,k-1}.$ Since $\pi _{v}\in \mathcal{P}_{k-2}$ and $%
\nabla ^{k-1}v\in (L_{s+1}^{p})^{\nu (k-1,N)}$ and since $s+1<-(k-1)<-1,$
the coefficients of the homogeneous part $\pi _{v,k-2}$ of degree $k-2$ of $%
\pi _{v}$ are given by (\ref{20}) with $u$ replaced with $(\alpha
!)^{-1}\partial ^{\alpha }v$ for $|\alpha |=k-2.$ The (unique) polynomial $%
\pi _{u}$ is fully determined after $k$ steps. Its definition is obviously
independent of $1\leq p<\infty $ such that $\nabla ^{k}u\in (L_{s}^{p})^{\nu
(k,N)}.$

If $s>-1,$ then $\mathcal{P}_{k-1}\subset L_{s+k}^{q}$ for every $q.$ In
particular, if $q\in I_{k,p},$ it follows from $u-\pi _{u}\in L_{s+k}^{q}$
that $u\in L_{s+k}^{q}.$ There are now many different ways to define a
suitable polynomial $\pi _{u}.$ By Theorem \ref{th3}, a possible choice for
the coefficients $(\alpha !)^{-1}c_{\alpha }$ of $\pi _{u,k-1}$ in (\ref{31}
) is $(\alpha !)^{-1}c_{\alpha }=|B_{\rho }|^{-1}\int_{B_{\rho }}(\alpha
!)^{-1}\partial ^{\alpha }u$ where $|\alpha |=k-1$ and $\rho >0$ is
independent of $u,$ but there are other options. Indeed, one could as well
define $(\alpha !)^{-1}c_{\alpha }=|B_{\rho _{\alpha
}}|^{-1}\int_{B(x_{\alpha },\rho _{\alpha })}(\alpha !)^{-1}\partial
^{\alpha }u,$ where $x_{\alpha }$ and $\rho _{\alpha }>0$ are independent of 
$u$ but depend upon $\alpha ;$ see the comments after Theorem \ref{th3}.

Once $\pi _{u,k-1}$ has been chosen, $v=u-\pi _{u,k-1}$ is known and the
problem is reduced to finding $\pi _{v}.$ This is the same problem with $s$
replaced with $s+1$ and $k$ replaced with $k-1.$ Since $s>-1$ implies $%
s+1>-1,$ the coefficients of the homogeneous part $\pi _{v,k-2}$ of $\pi
_{v} $ can be defined by $(\alpha !)^{-1}c_{\alpha }=|B_{\rho _{\alpha
}}|^{-1}\int_{B(x_{\alpha },\rho _{\alpha })}(\alpha !)^{-1}\partial
^{\alpha }v$ where $|\alpha |=k-2$ and $x_{\alpha }$ and $\rho _{\alpha }>0$
are once again arbitrarily chosen. A polynomial $\pi _{u}$ is obtained in $k$
steps. Clearly, different choices of $x_{\alpha }$ and $\rho _{\alpha }$
produce different polynomials $\pi _{u},$ but no matter how these choices
are made, they are always independent of $1\leq p<\infty $ such that $\nabla
^{k}u\in (L_{s}^{p})^{\nu (k,N)}.$

Still when $s>-1,$ but only when $p>N,$ the Taylor polynomial of $u$ at any
point $x_{0}$ is another possible choice for $\pi _{u}.$ This is most easily
seen when $x_{0}=0.$ First, by Remark \ref{rm4}, $(\alpha !)^{-1}\partial
^{\alpha }u(0)$ is a possible choice for the coefficients $(\alpha
!)^{-1}c_{\alpha }$ of $\pi _{u,k-1}$ when $|\alpha |=k-1.$ Next, $(\alpha
!)^{-1}\partial ^{\alpha }v(0)$ is a possible choice for the coefficients of 
$\pi _{v,k-2}$ when $|\alpha |=k-2,$ but since $v=u-\pi _{u,k-1}$ and $%
\partial ^{\alpha }\pi _{u,k-1}(0)=0$ when $|\alpha |=k-2,$ these
coefficients are just $(\alpha !)^{-1}\partial ^{\alpha }u(0).$ By repeating
this argument, $\sum_{|\alpha |\leq k-1}(\alpha !)^{-1}\partial ^{\alpha
}u(0)x^{\alpha }$ is a possible choice for $\pi _{u}$ and, by changing $u(x)$
into $u(x+x_{0}),$ it follows that $\sum_{|\alpha |\leq k-1}(\alpha
!)^{-1}\partial ^{\alpha }u(x_{0})(x-x_{0})^{\alpha }$ is an equally
possible choice.

When $s\in (-k,-1),$ the procedure to find a suitable polynomial $\pi _{u}$
combines the approaches of the previous two cases. Because $s<-1,$ the
homogeneous part $\pi _{u,k-1}$ of degree $k-1$ of $\pi _{u}$ is still
unique, but the homogeneous part $\pi _{v,k-2}$ of degree $k-2$ of $\pi _{v}$
is unique only if $s+1<-1.$ Otherwise, it must be determined as indicated
above when $s>-1$ after replacing $s$ with $s+1.$ More generally, if $%
k_{s}:=E(s+k+1)$ where $E$ denotes integer part, so that $1\leq k_{s}\leq
k-1,$ the part of $\pi _{u}$ of degree greater than or equal to $k_{s}$ is
unique and the part of degree less than or equal to $k_{s}-1$ is determined
as indicated above when $s>-1$ after replacing $s$ with $s+k-k_{s}>-1.$ Once
again, the calculation of $\pi _{u}$ does not depend on $1\leq p<\infty $
such that $\nabla ^{k}u\in (L_{s}^{p})^{\nu (k,N)}.$

The next corollary singles out the more familiar case when $s=-N/p,$ i.e.,
when $\nabla ^{k}u\in (L^{p})^{\nu (k,N)}.$ Since $s$ and $p$ are now
related, $s$ lies in some connected component of $\Bbb{R}\backslash
\{-k,...,-1\}$ if and only if $p$ lies in the corresponding connected
component of $\Bbb{R}\backslash \{N/k,N/(k-1),...,N\}.$

\begin{corollary}
\label{cor10}Suppose that $k\in \Bbb{N}$ and that $1\leq p<\infty $ with $%
p\neq N/j$ for $j=1,...,k.$ If $u\in \mathcal{D}^{\prime }$ and $\nabla
^{k}u\in (L^{p})^{\nu (k,N)},$ there is a polynomial $\pi _{u}\in \mathcal{P}%
_{k-1}$ independent of $p$ in each connected component of $\Bbb{R}\backslash
\{N/k,N/(k-1),...,N\},$ unique if $p<N/k,$ such that $\nabla ^{k-j}(u-\pi
_{u})\in (L_{j-N/p}^{q})^{\nu (k-j,N)}$ for every $1\leq j\leq k$ and every $%
q\in I_{j,p}$ and there is a constant $C=C(s,j,p,q)>0$ independent of $u$
such that 
\begin{equation}
||\,|\nabla ^{k-j}(u-\pi _{u})|\,||_{L_{j-N/p}^{q}}\leq C||\,|\nabla
^{k}u|\,||_{p}.  \label{34}
\end{equation}
Furthermore, $\lim_{|x|\rightarrow \infty }|x|^{-j+N/p}(\nabla
^{k-j}u(x)-\nabla ^{k-j}\pi _{u}(x))=0$ if $1\leq j\leq k$ and $p>N/j.$ (In
particular, $\lim_{|x|\rightarrow \infty }|x|^{-j+N/p}\nabla ^{k-j}u(x)=0$
for every $1\leq j\leq k$ if $p>N.$)
\end{corollary}

When $j=k,$ the ``furthermore'' part of Corollary \ref{cor10} was proved by
Mizuta \cite{Mi86}. If $p<N/k,$ (\ref{34}) for $q=p^{*j}=Np/(N-jp)$ and $%
j\in \{1,...,k\}$ becomes $||\,|\nabla ^{k-j}(u-\pi _{u})|\,||_{p^{*}}\leq
C||\,|\nabla ^{k}u|\,||_{p}$ for $1\leq j\leq k,$ a Sobolev inequality of
order $j$ easily proved directly by induction. It can be found in \cite{Ga11}%
, in the more general form given in Remark \ref{rm6} (i.e., when $p<N/\ell $
for some $\ell \leq k$ and $1\leq j\leq \ell $).

The scaling trick used when $k=1$ yields numerous inequalities (\ref{28})
with $1+|x|$ replaced with $|x|.$ We only give a small sample when $s=-N/p$
(so that (\ref{28}) is (\ref{34})) and $j=k.$ The earlier discussion about
the calculation of $\pi _{u}$ is crucial. If $k<N$ and $p<N/k,$ then $%
s=-N/p<-k$ and, with $c_{\partial ^{\alpha }u}$ given by (\ref{20}) for $%
\partial ^{\alpha }u,$ it follows that $\pi _{u}=0$ if $c_{\partial ^{\alpha
}u}=0$ for $|\alpha |\leq k-1$ (in particular, if $u$ has compact support).
By scaling (\ref{28}) with $j=k,$ we get $||\,|x|^{-k+N/p-N/q}u||_{q}\leq
C||\,|\nabla ^{k}u|\,||_{p}$ for $q\in I_{k,p},$ a Hardy-Sobolev inequality
of order $k.$

If ($p\geq 1$ and) $N/\ell <p<N/(\ell -1)$ for some $1\leq \ell \leq k$ and
with $c_{\partial ^{\alpha }u}$ given by (\ref{20}) for $\partial ^{\alpha
}u,$ then $\deg \pi _{u}\leq k-\ell $ if $c_{\partial ^{\alpha }u}=0$ when $%
k-\ell +1\leq |\alpha |\leq k-1$ (vacuous if $\ell =1$ and trivially true if 
$\ell \geq 2$ and $\nabla ^{k-\ell +1}u$ has compact support) and $\pi _{u}$
may be chosen as the Taylor polynomial of $u$ of order $k-\ell $ at $0.$
Then, scaling in (\ref{28}) with $j=k$ yields $||\,|x|^{-k+N/p-N/q}(u-\pi 
_{u})||_{q}\leq C||\,|\nabla ^{k}u|\,||_{p}$ for $q\in I_{k,p}=[p,\infty ].$
If $\ell =1$ (i.e., $p>N$), $\pi _{u}$ is the Taylor polynomial of $u$ of
order $k-1$ at $0$ and the inequality is a Hardy (Morrey) inequality of
order $k$ if $q=p$ ($q=\infty $). If $\ell >1,$ the required conditions are
non-standard in inequalities of this sort.\bigskip 

\section{An application to embeddings of weighted Sobolev spaces\label%
{application}}

For $k\in \Bbb{N},s\in \Bbb{R}$ and $1\leq p<\infty ,1\leq q\leq \infty ,$
consider the space 
\begin{equation*}
W_{s}^{k,q,p}:=\{u\in L_{s+k}^{q}:\nabla ^{k}u\in (L_{s}^{p})^{\nu (k,N)}\},
\end{equation*}
equipped with the Banach space norm $||u||_{L_{s+k}^{q}}+||\,|\nabla
^{k}u|\,||_{L_{s}^{p}}.$ When $k=1,s<-1$ and $q\in I_{1,p}$ is finite, it
follows from \cite[Example 21.10 (i), p. 309]{OpKu90} that $%
W_{s}^{1,p,p}\hookrightarrow W_{s}^{1,q,p}$ and that the norm of $%
W_{s}^{1,p,p}$ is equivalent to the norm $||\,|\nabla u|\,||_{L_{s}^{p}}.$
In Theorem \ref{th12} below, we show that if $s\neq -1$ -not just $s<-1$-
and $q\in I_{1,p}$ (possibly $\infty $), then in fact $%
W_{s}^{1,p,p}=W_{s}^{1,q,p}$ with equivalent norms\footnote{%
However, the norms are equivalent to $||\,|\nabla u|\,||_{L_{s}^{p}}$ only
when $s<-1.$} and that the reverse embedding $W_{s}^{1,q,p}\hookrightarrow
W_{s}^{1,p,p}$ holds for every $1\leq q\leq \infty .$ We also show that the
spaces $W_{s}^{k,q,p},$ $k\in \Bbb{N},$ have similar (and other) properties
if $s\notin \{-k,...,-1\}.$

\begin{lemma}
\label{lm11}Suppose $k\in \Bbb{N}$ and $s\in \Bbb{R},$ $s\notin
\{-k,...,-1\}.$ For a polynomial $\pi \in \mathcal{P}_{k-1},$ the following
conditions are equivalent:\newline
(i) $\pi \in L_{s+k}^{q}$ for every $1\leq q\leq \infty .$\newline
(ii) $\pi \in L_{s+k}^{q_{1}}+L_{s+k}^{q_{2}}$ for some $1\leq q_{1},q_{2}%
\leq \infty .$
\end{lemma}

\begin{proof}
(i) $\Rightarrow $ (ii) is obvious.

(ii) $\Rightarrow $ (i). With no loss of generality, assume $\pi \neq 0$ and
let $d:=\deg \pi \leq k-1.$ Since $\pi \in L_{s+k}^{q}$ for every $1\leq
q\leq \infty $ if $d<s+k,$ it suffices to prove that, indeed, $d<s+k.$

We shall use the remark that $|\pi (x)|$ grows (pointwise) as fast as $%
|x|^{d}$ on some sector (open cone with vertex at the origin) $\Sigma .$
Specifically, there are $\delta >0$ and $R>0$ such that $|\pi (x)|\geq
\delta (1+|x|)^{d}$ for $x\in \Sigma _{R}:=\{x\in \Sigma :|x|\geq R\}.$

Write $\pi =f+g$ with $f\in L_{s+k}^{q_{1}}$ and $g\in L_{s+k}^{q_{2}}.$
Since $|\pi |\leq |f|+|g|,$ it follows that if $x\in \Sigma _{R},$ then
either $|f(x)|\geq (\delta /2)(1+|x|)^{d}$ or $|g(x)|\geq (\delta
/2)(1+|x|)^{d}.$ Set $E_{f}:=\{x\in \Bbb{R}^{N}:|f(x)|\geq (\delta
/2)(1+|x|)^{d}\}.$ If $q_{1}<\infty ,$ then $%
\int_{E_{f}}(1+|x|)^{-(s+k-d)q_{1}-N}dx\leq (2/\delta
)^{q_{1}}||f||_{L_{s+k}^{q_{1}}}^{q_{1}}<\infty ,$ i.e., $E_{f}$ has finite $%
w_{1}(x)dx$ measure, where $w_{1}(x):=(1+|x|)^{-(s+k-d)q_{1}-N}.$ Likewise,
if $E_{g}:=\{x\in \Bbb{R}^{N}:|g(x)|\geq (\delta /2)(1+|x|)^{d}\}$ and $%
q_{2}<\infty ,$ then $E_{g}$ has finite $w_{2}(x)dx$ measure where $%
w_{2}(x):=(1+|x|)^{-(s+k-d)q_{2}-N}.$ Thus, both $E_{f}$ and $E_{g}$ have
finite $w_{0}(x)dx$ measure, where $w_{0}(x):=(1+|x|)^{-(s+k-d)q_{0}-N}$ and 
$q_{0}=q_{1}$ or $q_{0}=q_{2},$ depending upon which of the two exponents $%
-(s+k-d)q_{i}-N,i=1,2,$ is smaller.

As a result, $\Sigma _{R}\subset E_{f}\cup E_{g}$ has finite $w_{0}(x)dx$
measure $\int_{\Sigma _{R}}(1+|x|)^{-(s+k-d)q_{0}-N}dx$ and a calculation in
spherical coordinates shows at once that this happens if and only if $d<s+k.$

Suppose now that $q_{1}=\infty ,$ so that $(1+|x|)^{d}(1+|x|)^{-(s+k)}$ is
bounded on $E_{f}.$ This can only happen if $E_{f}$ is bounded or if $d\leq
s+k.$ In the latter case, $d<s+k$ since $0\leq d\leq k-1$ and $s\notin
\{-k,...,-1\}.$ Assume then that $E_{f}$ is bounded. If $q_{2}<\infty ,$ the
result that $E_{g}$ has finite $w_{2}(x)dx$ measure continues to hold and
then the same thing is true of $E_{f}$ (bounded). Hence, $\Sigma _{R}$ has
finite $w_{2}(x)dx$ measure and the same argument as before yields $d<s+k.$
By symmetry, this remains true if $q_{2}=\infty $ and $q_{1}<\infty .$

Lastly, suppose that $q_{1}=q_{2}=\infty .$ Since $\Sigma _{R}\subset
E_{f}\cup E_{g}$ is unbounded, one at least among $E_{f}$ and $E_{g}$ is
unbounded and so, as was seen above, $d<s+k.$
\end{proof}

\begin{remark}
\label{rm7}Lemma \ref{lm11} is also true, with the same proof, on $\Bbb{R}%
_{\pm }.$
\end{remark}

\begin{theorem}
\label{th12} Suppose that $k\in \Bbb{N},$ that $1\leq p<\infty $ and that $
s\in \Bbb{R},s\notin \{-k,...,-1\}$ (it is not assumed that $N>1$ if $s<-1$%
). Define $k_{s}:=k$ if $s>-1,k_{s}:=E(s+k+1)$ (integer part) if $s\in
(-k,-1)$ and $k_{s}:=0$ if $s<-k.$ Then, \newline
(i) $W_{s}^{k,q,p}\hookrightarrow W_{s}^{k,p,p}$ for every $1\leq q\leq
\infty $. \newline
(ii) If $u\in W_{s}^{k,p,p},$ then $\nabla ^{k-j}u\in (L_{s+j}^{q})^{\nu
(k-j,N)}$ for every $1\leq j\leq k$ and every $q\in I_{j,p}$ (see (\ref{4}))
and there is a constant $C>0$ independent of $u$ such that 
\begin{equation}
||\,|\nabla ^{k-j}u|\,||_{L_{s+j}^{q}}\leq C||\,|\nabla
^{k}u|\,||_{L_{s}^{p}}\text{ if }1\leq j\leq k-k_{s}  \label{35}
\end{equation}
and that 
\begin{equation}
||\,|\nabla ^{k-j}u|\,||_{L_{s+j}^{q}}\leq C(||u||_{L_{s+k}^{p}}+||\,|\nabla
^{k}u|\,||_{L_{s}^{p}}).\text{ if }1\leq j\leq k.  \label{36}
\end{equation}
(iii) $W_{s}^{k,q,p}=W_{s}^{k,p,p}$ for every $q\in I_{k,p},$ with
equivalent norms as $q$ is varied. Furthermore, if $s<-k,$ the norm of $%
W_{s}^{k,p,p}$ is equivalent to $||\,|\nabla ^{k}u|\,||_{L_{s}^{p}}.$
\end{theorem}

\begin{proof}
(i) In a first step, we also assume that $N>1$ if $s<-1.$ If $u\in
W_{s}^{k,q,p}$ with $1\leq q\leq \infty ,$ then $u\in L_{s+k}^{q}$ and $%
\nabla ^{k}u\in (L_{s}^{p})^{\nu (k,N)}.$ By Theorem \ref{th9}, there is a
polynomial $\pi _{u}\in \mathcal{P}_{k-1}$ such that $u-\pi _{u}\in
L_{s+k}^{p}$ and, since $u\in L_{s+k}^{q},$ it follows that $\pi _{u}\in
L_{s+k}^{p}+L_{s+k}^{q}.$ Thus, $\pi _{u}\in L_{s+k}^{p}$ by (ii) $%
\Rightarrow $ (i) in Lemma \ref{lm11}, so that $u\in L_{s+k}^{p}.$ This
shows that $W_{s}^{k,q,p}\subset W_{s}^{k,p,p}.$

If $s<-1$ and $N=1,$ the above still shows, by Remarks \ref{rm5} and \ref
{rm7}, that $W_{s}^{k,q,p}(\Bbb{R}_{\pm })\subset W_{s}^{k,p,p}(\Bbb{R}_{\pm
})$ for every $1\leq q\leq \infty .$ Thus, if $u\in W_{s}^{k,q,p}(\Bbb{R})$
with $1\leq q\leq \infty ,$ then $u\in W_{loc}^{k,p}(\Bbb{R})$ and $u\in
W_{s}^{k,p,p}(\Bbb{R}_{-})\cup W_{s}^{k,p,p}(\Bbb{R}_{+}),$ whence $u\in
W_{s}^{k,p,p}(\Bbb{R}).$ This shows that $W_{s}^{k,q,p}\subset W_{s}^{k,p,p}$
for every $1\leq q\leq \infty $ still holds when $N=1$.

That the above embeddings are continuous follows from the closed graph
theorem, for if $u_{n}\rightarrow u$ in $W_{s}^{k,q,p}$ and $%
u_{n}\rightarrow v$ in $W_{s}^{k,p,p},$ then $u_{n}$ tends to both $u$ and $
v $ in $\mathcal{D}^{\prime },$ so that $u=v.$

(ii) If $s<-k,$ so that $k_{s}=0,$ the polynomial $\pi _{u}$ of Theorem \ref
{th9} (when $N>1$) is unique, whence $\pi _{u}=0$ if $u\in
W_{s}^{k,p,p}\subset L_{s+k}^{p}$ and then (\ref{35}) follows from (\ref{28}%
). If $N=1,$ use the same arguments on $\Bbb{R}_{\pm }$ (Remark \ref{rm5}).
Evidently, (\ref{36}) follows from (\ref{35}).

Suppose now that $s\in (-k,-1),$ so that $s$ is not an integer and $%
k_{s}=E(s+k+1)<s+k+1.$ Assume $N>1.$ From the discussion after Theorem \ref
{th9}, the part of $\pi _{u}$ of degree greater than or equal to $k_{s}$ is
unique. Therefore, if $u\in W_{s}^{k,p,p}\subset L_{s+k}^{p},$ this part is $%
0,$ so that $\pi _{u}\in \mathcal{P}_{k_{s}-1}$ and (\ref{35}) follows from (%
\ref{28}) since $\nabla ^{k-j}\pi _{u}=0$ when $1\leq j\leq k-k_{s}.$

To prove (\ref{36}) (obvious from (\ref{35}) when $1\leq j\leq k-k_{s}$), we
first show that $\nabla ^{k-j}u\in (L_{s+j}^{q})^{\nu (k-j,N)}$ for $1\leq
j\leq k.$ By Theorem \ref{th9}, $\nabla ^{k-j}u-\nabla ^{k-j}\pi _{u}\in
(L_{s+j}^{q})^{\nu (k-j,N)}$ and, since $\pi _{u}\in \mathcal{P}_{k_{s}-1},$
it follows that $\nabla ^{k-j}\pi _{u}\in (\mathcal{P}_{k_{s}-1-k+j})^{\nu
(k-j,N)}.$ But $\mathcal{P}_{k_{s}-1-k+j}\subset L_{s+j}^{q}$ since $%
s+j>k_{s}-1-k+j$ (recall $k_{s}<s+k+1$) and so $\nabla ^{k-j}u\in
(L_{s+j}^{q})^{\nu (k-j,N)},$ as claimed.

This shows that $W_{s}^{k,p,p}\subset W_{s+j}^{k-j,p,q}.$ By the closed
graph theorem, the embedding is continuous and (\ref{36}) follows. As
before, if $N=1,$ repeat the same arguments on $\Bbb{R}_{\pm }.$

Suppose now that $s>-1,$ so that $k_{s}=k$ and (\ref{35}) is trivial. The
proof of (\ref{36}) proceeds as above, based on the remark that $\nabla
^{k-j}\pi _{u}\in (\mathcal{P}_{j-1})^{\nu (k-j,N)}$ and that $\mathcal{P}
_{j-1}\subset L_{s+j}^{q}$ since $s+j>j-1.$

(iii) By (i), $W_{s}^{k,q,p}\hookrightarrow W_{s}^{k,p,p}$ and, by (ii) with 
$j=k,$ $W_{s}^{k,p,p}\subset W_{s}^{k,q,p}$ if $q\in I_{k,p}.$ The
continuity of the embedding (hence the equivalence of norms) follows from (%
\ref{36}) with $j=k.$ If also $s<-k,$ then $k_{s}=0$ and, by (\ref{35}) with 
$j=k$ and $q=p\in I_{k,p},$ the norm of $W_{s}^{k,p,p}$ is equivalent with $%
|\,|\nabla ^{k}u|\,||_{L_{s}^{p}}.$
\end{proof}

\section{Exponential growth\label{exponential}}

If $s\in \Bbb{R}$ and $1\leq p\leq \infty ,$ define 
\begin{equation*}
L_{\exp ,s}^{p}:=\{u\in L_{loc}^{p}:e^{-s|x|}u\in L^{p}\},
\end{equation*}
equipped with the Banach space norm $||u||_{L_{\exp
,s}^{p}}:=||e^{-s|x|}u||_{p}.$ If $p<\infty ,$ $L_{\exp ,s}^{p}=L^{p}(\Bbb{R}
^{N};e^{-sp|x|}dx)$ and $L_{\exp ,0}^{p}=L^{p}=L_{-N/p}^{p}.$

It is readily checked that $u(x)=e^{t|x|}$ is in $L_{\exp ,s}^{p}$ for every 
$t<s$ and that $u(x)=(1+|x|)^{t}$ is in $L_{\exp ,s}^{p}$ for every $t\in 
\Bbb{\ R}$ if $s>0.$ Thus, when $p<\infty ,$ it is appropriate to say that
the functions of $L_{\exp ,s}^{p}$ grow slower than $e^{s|x|}$ at infinity
in the $L^{p}$ sense. The functions of $L_{\exp ,s}^{\infty }$ do not grow
faster than $e^{s|x|},$ and grow slower if $\lim_{R\rightarrow \infty }%
\limfunc{ess}\sup_{|x|>R}e^{-s|x|}|u|=0,$ or simply $\lim_{|x|\rightarrow
\infty }e^{-s|x|}u(x)=0$ if $u$ is continuous.

By Remark \ref{rm1}, Lemma \ref{lm1} (approximation by mollification) holds
in all the spaces $L_{\exp ,s}^{p}$ with $p<\infty .$ Also, the
one-dimensional Hardy inequalities of Lemmas \ref{lm2} and \ref{lm4} have
the following counterpart (\cite[Theorems 5.9 and 6.2]{OpKu90}):

\begin{lemma}
\label{lm13}Suppose that $1\leq p<\infty $ and that $s\neq 0.$ \newline
(i) If $s>0$ and $\rho >0,$ there is a constant $C>0$ such that 
\begin{equation}
\left( \int_{\rho }^{\infty }e^{-spr}r^{N-1}|f(r)-f(\rho )|^{p}dr\right)
^{1/p}\leq C\left( \int_{\rho }^{\infty }e^{-spr}r^{N-1}|f^{\prime
}(r)|^{p}dr\right) ^{1/p},  \label{37}
\end{equation}
for every locally absolutely continuous function $f$ on $[\rho ,\infty ).$%
\newline
(ii) If $s<0,$ there is a constant $C>0$ such that 
\begin{equation}
\left( \int_{0}^{\infty }e^{-spr}r^{N-1}|f(r)|^{p}dr\right) ^{1/p}\leq
C\left( \int_{0}^{\infty }e^{-spr}r^{N-1}|f^{\prime }(r)|^{p}dr\right)
^{1/p},  \label{38}
\end{equation}
for every absolutely continuous function $f$ on $(0,\infty )$ such that $%
\lim_{r\rightarrow \infty }f(r)=0.$
\end{lemma}

Both (\ref{37}) and (\ref{38}) remain true, but will not be needed, when $p$
is replaced with $q\in [p,\infty ]$ in the left-hand side, with the usual
modification when $q=\infty .$

\begin{theorem}
\label{th14}Suppose that $k\in \Bbb{N},$ that $s\neq 0$ and that $N>1$ if $%
s<0.$ Let $u\in \mathcal{D}^{\prime }$ be such that $\nabla ^{k}u\in
(L_{\exp ,s}^{p})^{\nu (k,N)}$ with $1\leq p<\infty .$ Then, there is a
polynomial $\pi _{u}\in \mathcal{P}_{k-1}$ independent of $s$ and $p,$
unique if $s<0,$ such that $\nabla ^{k-j}(u-\pi _{u})\in (L_{\exp
,s}^{q})^{\nu (k-j,N)}$ for every $1\leq j\leq k$ and every $q\in I_{j,p}$
and there is a constant $C=C(s,j,p,q)>0$ independent of $u$ such that 
\begin{equation*}
||\,|\nabla ^{k-j}(u-\pi _{u})|\,||_{L_{\exp ,s}^{q}}\leq C||\,|\nabla
^{k}u|\,||_{L_{\exp ,s}^{p}}.
\end{equation*}
Furthermore, $\lim_{|x|\rightarrow \infty }e^{-s|x|}(\nabla
^{k-j}u(x)-\nabla ^{k-j}\pi _{u}(x))=0$ if $p=N=j=1$ or if $1\leq j\leq k$
and $p>N/j.$ (In particular, $\lim_{|x|\rightarrow \infty }e^{-s|x|}\nabla
^{k-j}u(x)=0$ if also $s>0.$)
\end{theorem}

\begin{proof}
Suppose first that $k=1,$ so that $\pi _{u}$ is a constant $c_{u}.$ The case 
$q=p$ can be handled along the lines of the proof of Theorem \ref{th3} (when 
$s>0$) or Theorem \ref{th5} (when $s<0$), upon merely using Lemma \ref{lm13}
instead of Lemma \ref{lm2} or Lemma \ref{lm4}. We skip the details.

It follows from $\nabla u\in (L_{\exp ,s}^{p})^{N}$ and $u-c_{u}\in L_{\exp
,s}^{p}$ that $e^{-s|x|}(u-c_{u})\in W^{1,p}$ (use $\nabla
e^{-s|x|}=-se^{-s|x|}|x|^{-1}x$). Thus, by the Sobolev embedding theorem, $%
e^{-s|x|}(u-c_{u})\in L^{q},$ i.e., $u-c_{u}\in L_{\exp ,s}^{q},$ for every $%
q\in I_{1,p}.$ If $p>N,$ it is well-known that the functions of $W^{1,p}$
tend to $0$ at infinity, so that $\lim_{|x|\rightarrow \infty
}e^{-s|x|}(u(x)-c_{u})=0.$ If $s>0,$ then $\lim_{|x|\rightarrow \infty
}e^{-s|x|}c_{u}=0$ and so $\lim_{|x|\rightarrow \infty }e^{-s|x|}u(x)=0.$
This proves the theorem when $k=1.$ The general case follows by induction;
see the proof of Theorem \ref{th9}.
\end{proof}

When $s=0,$ Theorem \ref{th9} with $s=-N/p$ (that is, Corollary \ref{cor10})
must be substituted for Theorem \ref{th14}, at least when $p\neq
N/j,j=1,...,k.$

If $s<0$ and $k=1,$ the polynomial $\pi _{u}$ is a constant $c_{u}$ given by
(\ref{20})). If $s>0$ and $k=1,$ a formula for $c_{u}$ is $|B_{\rho
}|^{-1}\int_{B(x_{0},\rho )}u$ where $x_{0}\in \Bbb{R}^{N}$ and $\rho >0$
are chosen independent of $u.$ When $k\in \Bbb{N},$ these formulas can be
used to find the coefficients of $\pi _{u};$ see the comments after Theorem 
\ref{th9}.

There is also an analog of Theorem \ref{th12}, with an entirely similar
proof. Define 
\begin{equation*}
W_{\exp ,s}^{k,q,p}:=\{u\in L_{\exp ,s}^{q}:\nabla ^{k}u\in (L_{\exp
,s}^{p})^{\nu (k,N)}\},
\end{equation*}
with Banach space norm $||u||_{L_{\exp ,s}^{q}}+||\,|\nabla
^{k}u|\,||_{L_{\exp ,s}^{p}}.$

\begin{theorem}
\label{th15}Suppose that $k\in \Bbb{N},$ that $1\leq p<\infty $ and that $%
s\in \Bbb{R},s\neq 0$ (it is not assumed that $N>1$ if $s<0$). Then, \newline
(i) $W_{\exp ,s}^{k,q,p}\hookrightarrow W_{\exp ,s}^{k,p,p}$ for every $%
1\leq q\leq \infty $. \newline
(ii) If $u\in W_{\exp ,s}^{k,p,p},$ then $\nabla ^{k-j}u\in (L_{\exp
,s}^{q})^{\nu (k-j,N)}$ for every $1\leq j\leq k$ and every $q\in I_{j,p}$
and there is a constant $C>0$ independent of $u$ such that 
\begin{equation*}
||\,|\nabla ^{k-j}u|\,||_{L_{\exp ,s}^{q}}\leq C||\,|\nabla
^{k}u|\,||_{L_{\exp ,s}^{p}},
\end{equation*}
when $s<0$ and that 
\begin{equation*}
||\,|\nabla ^{k-j}u|\,||_{L_{\exp ,s}^{q}}\leq C(||u||_{L_{\exp
,s}^{p}}+||\,|\nabla ^{k}u|\,||_{L_{\exp s}^{p}}).
\end{equation*}
(iii) $W_{\exp ,s}^{k,q,p}=W_{\exp ,s}^{k,p,p}$ for every $q\in I_{k,p},$
with equivalent norms as $q$ is varied. Furthermore, if $s<0,$ the norm of $%
W_{\exp ,s}^{k,p,p}$ is equivalent to $||\,|\nabla ^{k}u|\,||_{L_{\exp
,s}^{p}}.$
\end{theorem}

\end{document}